\documentclass{artv3}
\usepackage{graphicx}
\usepackage[all]{xy}
\usepackage{fullpage}
\usepackage{amsmath,amsthm}
\usepackage{amssymb,amsfonts,amsbsy}
\usepackage{dsfont,mathrsfs}
\usepackage{color}
\usepackage[utf8]{inputenc}
\usepackage[T1]{fontenc}
\usepackage{enumitem}
\title{First integral cohomology group of the pure mapping class group of a non-orientable surface of infinite type}
\firstauthor{Jesús Hernández Hernández}
\date{}
\email{jhdez@matmor.unam.mx	}
\address{Centro de Ciencias Matemáticas, UNAM\newline
\indent Universidad Nacional Autónoma de México\newline
\indent Morelia, Mich. 58190\newline
\indent Mexico}
\urladdr{https://sites.google.com/site/jhdezhdez/}
\secondauthor{Cristhian E. Hidber}
\secondemail{hidber@matmor.unam.mx	}
\secondaddress{Centro de Ciencias Matemáticas, UNAM\newline
\indent Universidad Nacional Autónoma de México\newline
\indent Morelia, Mich. 58190\newline
\indent Mexico}
\secondurladdr{https://sites.google.com/view/cristhian-e-hidber}
\newtheorem{theorem}{Theorem}[section]
\newtheorem{lemma}[theorem]{Lemma}

\newtheorem{definition}[theorem]{Definition}
\newtheorem{corollary}[theorem]{Corollary}
\newtheorem{remark}[theorem]{Remark}

\newtheorem{Theo}{Theorem}

\newtheorem{Coro}[Theo]{Corollary}
\newcommand*{\fun}[3]{#1: #2 \rightarrow #3}
\newcommand*{\ccomp}[1]{\mathcal{C}(#1)}
\newcommand*{\adj}[1]{\mathcal{A}(#1)}
\newcommand*{\Mod}[1]{\mathrm{Map}(#1)}

\newcommand*{\PMod}[1]{\mathrm{PMap}(#1)}
\newcommand*{\PModc}[1]{\mathrm{PMap}_{c}(#1)}
\newcommand*{\PModcc}[1]{\overline{\mathrm{PMap}_{c}(#1)}}
\newcommand*{\Homeo}[1]{\mathrm{Homeo}(#1)}
\newcommand*{\ColonEqq}{\displaystyle \mathrel{\mathop:}=}
\newcommand*{\wrt}{with respect to }
\newcommand*{\id}{\mathrm{id}}
\newcommand*{\Aut}[1]{\mathrm{Aut}(#1)}
\newcommand*{\bnm}[1]{[\![#1]\!]}
\newcommand*{\prodhi}{\displaystyle \prod_{0 \leq i < r} \langle h_{i} \rangle}
\newcommand{\calC}{\mathcal C}
\newcommand{\calD}{\mathcal D}
\newcommand{\calH}{\mathcal H}

\DeclareMathOperator{\End}{\mathcal E}
\DeclareMathOperator{\Endg}{\End_{\infty}}
\DeclareMathOperator{\Endn}{\End_{-}}
\newcommand{\Z}{\ensuremath{\mathbb{Z}}}
\newcommand{\R}{\ensuremath{\mathbb{R}}}
\DeclareMathOperator\supp{supp}
\DeclareMathOperator{\Ker}{Ker}
\newcommand{\oshiftsup}{o\Sigma}
\newcommand{\nshiftsup}{n\Sigma}
\begin{document}
\maketitle
\begin{abstract}
 In this work we compute the first integral cohomology of the pure mapping class group of a non-orientable surface of infinite topological type and genus at least 3. To this purpose, we also prove several other results already known for orientable surfaces such as the existence of an Alexander method, the fact that the mapping class group is isomorphic to the automorphism group of the curve graph along with the topological rigidity of the curve graph, and the structure of the pure mapping class group as both a Polish group and a semi-direct product.\\[0.5cm]
\end{abstract}
%%%%%%%%%%%%%%%%%%%%%%%%%%%%%%%%%%%%%%%%%%%%%%%%%%%%%%%%%%%%
%%%%%%%%%%%%%%%%%%%%%%%%%%%%%%%%%%%%%%%%%%%%%%%%%%%%%%%%%%%%
%%%%%%%%%%%%%%%%%%%%%%%%%%%%%%%%%%%%%%%%%%%%%%%%%%%%%%%%%%%%
%%%%%%%%%%%%%%%%%%%%%%%%%%%%%%%%%%%%%%%%%%%%%%%%%%%%%%%%%%%%
%%%%%%%%%%%%%%%%%%%%%%%%%%%%%%%%%%%%%%%%%%%%%%%%%%%%%%%%%%%%

\section*{Introduction}

Let $N$ be a connected surface with empty boundary and define the mapping class group of $N$, denoted as $\Mod{N}$, as the group of isotopy classes of self-homeomorphisms of $N$. If $N$ is orientable, this is often called the extended mapping class group. The mapping class group has been studied for several decades now, and the most-commonly used tools for its study are the curves on the surface. This leads to the definition of the curve graph of $N$, denoted as $\ccomp{N}$: The curve graph is defined as the simplicial graph whose vertices are isotopy classes of essential simple closed curves on $N$. See Section \ref{sec1} for more details on these definitions. Given that $\Mod{N}$ naturally acts on $\ccomp{N}$, a lot of information and properties of $\Mod{N}$ have been obtained by studying this action.

If $N$ is an orientable surface of finite (topological) type, i.e. $N$ has finitely generated fundamental group, some of the results that have been proved are the following:
\begin{enumerate}
 \item There exists a collection of finitely many curves that completely determine a homeomorphism of $N$ up to isotopy. See Chapter 2 of \cite{FarbMargalit}.
 \item For all but finitely many surfaces, the action is rigid: every automorphism of $\ccomp{N}$ is induced by an element of $\Mod{N}$. See \cite{Ivanov}, \cite{Korkmaz}, \cite{Luo}.
 \item For all but finitely many surfaces, the first integral homology groups are finite groups depending on the surface. This implies that the first integral cohomology groups are trivial. See \cite{FarbMargalit} and the reference therein.
\end{enumerate}

If $N$ is a non-orientable surface of finite type, the same results are valid. However, in most of these cases either slight modifications of the proofs are needed as in (1), or whole new proofs are needed as in (2) and (3), and as such different papers have been dedicated to these results. See \cite{Paris}, \cite{Atalan-Korkmaz} and \cite{Stukow2010}.

On the other hand, for infinite-type surfaces there is a lot of work left to be done, particularly on the non-orientable case.

If $N$ is an orientable surface of infinite type, the same results as above have been proved recently: In \cite{HMV1} it is proved that there exists a locally finite collection of curves on $N$ that determine a homeomorphism up to isotopy. In \cite{HMV2} and \cite{BavardDowdallRafi} the respective authors proved independently that there is action rigidity along with topological rigidity. In \cite{AramayonaPatelVlamis} the respective authors compute the first integral cohomology group of the pure mapping class group (the subgroup of $\Mod{N}$ that acts trivially on the ends of the surface, denoted by $\PMod{N}$; see Section \ref{sec1} for more details) for the case that $N$ has genus at least two.

In this work we prove the analogous results. As mentioned before, many of the techniques used in this work are analogous to the orientable case. That said, in several cases the non-orientable nature of the surface forces the proofs to be different (particularly for the computation of the first integral cohomology group).

As such, the main results of our work are the following.

\begin{Theo}[Alexander Method]\label{Thm-AlexMethod}
 Let $N$ be a (possibly non-orientable) connected surface of infinite topological type. There exists a locally finite collection of essential, simple, closed curves $\Gamma = \{\gamma_{i}\}_{0 \leq i < \omega}$ that satisfies the following: If $h \in \mathrm{Homeo}(N)$ is such that for all $i \geq 0$, $h(\gamma_{i})$ is isotopic to $\gamma_{i}$, then $h$ is isotopic to the identity.
\end{Theo}

This theorem is the analogue of Theorem 1.1 in \cite{HMV1}. For the proof, we delegate the case when $N$ is orientable to the aforementioned theorem, and focus only on the non-orientable case. That said, the proof in this case is analogous, as such we only sketch the proofs of the related lemmata and theorems, while highlighting the differences.

\begin{Theo}\label{Thm-Atalan-Korkmaz}
 Let $N_{1}$ and $N_{2}$ be two connected (possibly non-orientable) surfaces of infinite topological type, and let $\varphi: \ccomp{N_{1}} \to \ccomp{N_{2}}$ be an isomorphism. Then $N_{1}$ is homeomorphic to $N_{2}$, and $\varphi$ is induced by a homeomorphism $N_{1} \to N_{2}$.
\end{Theo}

This theorem is the analogue of Theorem 1.1 in \cite{HMV2} and Theorem 1.3 in \cite{BavardDowdallRafi}. The proof is analogue to the proof of Theorem 1.1 in \cite{HMV2}, and as such we only sketch the proofs of the related lemmata and theorems, while highlighting the differences both in arguments and in references needed.

Using Theorem \ref{Thm-AlexMethod} and Theorem \ref{Thm-Atalan-Korkmaz}, we obtain the following classical corollary.

\begin{Coro}\label{Cor-Atalan-Korkmaz}
 Let $N$ be a connected (possibly non-orientable) surface of infinite topological type. Then the natural map $\Psi: \Mod{N} \to \Aut{\ccomp{N}}$ is an isomorphism.
\end{Coro}

This corollary is the analogue of Theorem 1.2 in \cite{HMV2}.

Now, as in the orientable case, $\Mod{N}$ has a natural topology which makes it a topological group: We equip $\mathrm{Homeo}(N)$ with the compact-open topology and then $\Mod{N}$ has the quotient of said topology.

On the other hand, Corollary \ref{Cor-Atalan-Korkmaz} tells us that pulling the permutation topology of $\Aut{\ccomp{N}}$, we can endow $\Mod{N}$ with a topology which makes it a Polish (separable and completely metrizable) topological group.

Using the same arguments as in the orientable case (see \cite{AramayonaPatelVlamis} and \cite{AramayonaVlamis}), we can see these two topologies coincide, and as such we have the following corollary.

\begin{Coro}\label{Cor-Top-Group-Iso}
 Let $N$ be a connected surface of infinite topological type. Arming $\Mod{N}$ with the quotient of the compact-open topology, and $\Aut{\ccomp{N}}$ with the permutation topology, then the natural map $\Psi: \Mod{N} \to \Aut{\ccomp{N}}$ is an isomorphism of topological groups. In particular, $\Mod{N}$ is a Polish group with the compact-open topology.
\end{Coro}

Then, to compute the first cohomology group of the pure mapping class group of $N$, we follow the ideas of \cite{AramayonaPatelVlamis}, and thus we need to understand more the topology and the topological generators of $\PMod{N}$. For this we need to recall some definitions.

A handle-shift is in essence taking an infinite strip connecting two ends of the surface with genus, and shifting said genus by one. The precise definition is given in Subsection \ref{section:handleshifts}. 

The \textit{compactly supported mapping class group}, denoted by $\PModc{N}$, is the subgroup of $\Mod{N}$ composed of the mapping classes that have representatives with compact support.

\begin{Theo}\label{Thm-PMod-TopGen-hi}
 Let $N$ be a connected (possibly non-orientable) surface of infinite topological type. If $N$ has at most one end accumulated by genus, then $\PMod{N} = \overline{\PModc{N}}$. If $N$ has at least two ends accumulated by genus, then there exist a constant $1 \leq r \leq \omega$ and a collection $\{h_{i}\}_{0 \leq i < r}$ of handle-shifts, such that $\PMod{N} = \overline{\langle \PModc{N}, \{h_{i}\}_{0 \leq i < r} \rangle}$ and $\overline{\langle \{h_{i}\}_{0 \leq i < r} \rangle}$ is isomorphic to $\Z^{r}$ as topological groups.
\end{Theo}

The general outline of the proof of this theorem is to follow the proof of Theorem 4 in \cite{PatelVlamis}, and prove that the compactly supported mapping class group and the set of all handle-shifts topologically generate the pure mapping class group. Then, we refine that result with the use of the collection $\{h_{i}\}_{0 \leq i < r}$. This collection is obtained via a basis of $H_{1}^{sep}(\widehat{N};\Z)$, and as such $r$ is the dimension of $H_{1}^{sep}(\widehat{N};\Z)$, where $\widehat{N}$ is the surface obtained from $N$ by ``forgetting'' all the planar ends of $N$, and $H_{1}^{sep}(\, \cdot \,; \Z)$ is the subgroup of $H_{1}(\, \cdot \,; \Z)$ generated by the homology classes that can be represented by separating simple closed curves.

One of the main difference between $\{h_{i}\}_{0 \leq i < r}$ and the similar collection obtained in Theorem 3 in \cite{AramayonaPatelVlamis}, is that in $N$ we cannot use for our proof any basis of $H_{1}^{sep}(\widehat{N};\Z)$. For example, for some surface we can produce a basis for $H_{1}^{sep}(\widehat{N};\Z)$ such that it is not clear how to topologically generate all the handle-shifts of $N$; see Subsection \ref{subsec:definehi} for more details. Thus, we construct what we call a ``good basis'' for $H_{1}^{sep}(\widehat{N};\Z)$, which in turn produces the collection $\{h_{i}\}_{0 \leq i <r}$ that satisfies the theorem. Also, due to the conclusion of Theorem \ref{Thm-PMod-TopGen-hi} we denote $\overline{\langle \{h_{i}\}_{0 \leq i <r} \rangle}$ by $\prodhi$ to emphasize the fact that this group is isomorphic to $\Z^{r}$ as topological groups.

A direct consequence of Theorem \ref{Thm-PMod-TopGen-hi} and the well-known fact that closed subgroups of Polish groups are Polish, is the following corollary.

\begin{Coro}\label{Cor-Polish}
 Let $N$ be a connected (possibly non-orientable) surface of infinite topological type. Then $\Mod{N}$, $\PMod{N}$ and $\PModcc{N}$ are Polish groups with their respective topologies.
\end{Coro}

Now, using the previous results we can define a homomorphism from $\PMod{N}$ to $\prodhi$, which we use to give a semi-direct product structure to $\PMod{N}$.

\begin{Theo}\label{Thm-Semi-direct-prod}
 Let $N$ be a connected (possibly non-orientable) surface of infinite topological type with at least two ends accumulated by genus. Then, we have that: $$\PMod{N} = \PModcc{N} \rtimes \prodhi.$$
\end{Theo}

Finally, using Theorems \ref{Thm-PMod-TopGen-hi} and \ref{Thm-Semi-direct-prod}, along with the results from Stukow in \cite{Stukow2010}, Dudley in \cite{Dudley1961}, Specker in \cite{Specker} and Blass and Göbel in \cite{BlassGoebel} we obtain the following corollary.

\begin{Coro}\label{Thm-First-cohom-grp}
 Let $N$ be a connected (possibly non-orientable) surface of infinite topological type with genus at least $3$. If $N$ has at most one end accumulated by genus, then $H^{1}(\PMod{N};\Z)$ is trivial. If $N$ has at least two ends accumulated by genus, then $\displaystyle H^{1}(\PMod{N}; \Z) = H^{1}(\Z^{r}; \Z) = \bigoplus_{0 \leq i < r} \Z$.
\end{Coro}

%%%%%%%%%%%%%%%%%%%%%%%%%%%%%%%%%%%%%%%%%%%%%%%%%%%%%%%%%%%%
%%%%%%%%%%%%%%%%%%%%%%%%%%%%%%%%%%%%%%%%%%%%%%%%%%%%%%%%%%%%
%%%%%%%%%%%%%%%%%%%%%%%%%%%%%%%%%%%%%%%%%%%%%%%%%%%%%%%%%%%%
%%%%%%%%%%%%%%%%%%%%%%%%%%%%%%%%%%%%%%%%%%%%%%%%%%%%%%%%%%%%
%%%%%%%%%%%%%%%%%%%%%%%%%%%%%%%%%%%%%%%%%%%%%%%%%%%%%%%%%%%%

\textbf{Acknowledgements:} 
 The first author was supported during the creation of this article by the research project grants UNAM-PAPIIT IA104620 and UNAM-PAPIIT IN102018. 
 The second author received support from a CONACYT Posdoctoral Fellowship and from UNAM-PAPIIT-IN105318.
 Both authors were supported during the creation of this article by the CONACYT Ciencia de Frontera 2019 research project grant CF 217392.
 Both authors would also like to thank Ulises A. Ramos-García for his thoughtful comments on this work.
 
%%%%%%%%%%%%%%%%%%%%%%%%%%%%%%%%%%%%%%%%%%%%%%%%%%%%%%%%%%%%
%%%%%%%%%%%%%%%%%%%%%%%%%%%%%%%%%%%%%%%%%%%%%%%%%%%%%%%%%%%%
%%%%%%%%%%%%%%%%%%%%%%%%%%%%%%%%%%%%%%%%%%%%%%%%%%%%%%%%%%%%
%%%%%%%%%%%%%%%%%%%%%%%%%%%%%%%%%%%%%%%%%%%%%%%%%%%%%%%%%%%%
%%%%%%%%%%%%%%%%%%%%%%%%%%%%%%%%%%%%%%%%%%%%%%%%%%%%%%%%%%%%

\section{Preliminaries}\label{sec1}

A \textit{curve} is a topological embedding of the unit circle into $N$. We often abuse notation and call ``curve'' the embedding, its image on $N$ or its isotopy class. The context makes clear which use we mean.

A curve is \textit{essential} if it is not isotopic to a boundary curve and if it does not bound a disk, a punctured disk or a Möbius band. Unless otherwise stated, all curves are assumed to be essential.

The \textit{(geometric) intersection number} of two isotopy classes of essential curves $\alpha$ and $\beta$ is defined as: $$i(\alpha,\beta):= \min \{|a \cap b| : a \in \alpha, b \in \beta\}.$$

We say two curves $\alpha$ and $\beta$ are in minimal position if $\alpha \cap \beta = i([\alpha],[\beta])$.

It is a well-known result (see \cite{FarbMargalit}) that if $N$ is doted with a hyperbolic metric, then in every isotopy class of a curve, there exists a unique geodesic representative. Also (see \cite{FarbMargalit}), any two geodesic representatives are in minimal position.

A set of curves $\mathcal{G}$ is \textit{locally finite} if for every compact subset $K$, the set $\{\alpha \in \mathcal{G}: \alpha \cap K \neq \varnothing\}$ is finite. A set of isotopy classes of curves $\Gamma$ is \textit{locally finite} if there exists a set of representatives $\mathcal{G}$ that is locally finite.

A \textit{multicurve} is a locally finite set of pairwise disjoint and pairwise non-isotopic curves. We often abuse notation and call ``multicurve'' the set of curves, their images on $N$ or its set of isotopy classes. The context makes clear which use we mean.

In this work, unless otherwise stated, by a \textit{subsurface $\Sigma$ of $N$} we mean a closed subsurface of $N$ such that every connected component of $\partial \Sigma$ is compact, and the natural inclusion $\Sigma \hookrightarrow N$ is $\pi_{1}$-injective.

A curve $\alpha$ is \textit{separating} if $N \setminus \alpha$ is disconnected. It is \textit{non-separating} otherwise.

Now we are ready for the following definition.

\begin{definition}
An increasing sequence of subsurfaces $\Sigma_{0} \subset \Sigma_{1} \subset \cdots \subset N$ is a \textit{principal exhaustion} if it satisfies the following:
\begin{enumerate}
 \item For each $j \geq 0$, $\Sigma_{j}$ is a finite-type subsurface such that each of its boundary curves are essential curves in $N$.
 \item For each $j \geq 0$, every connected component of $N \setminus \Sigma_{j}$ is an infinite type surface.
 \item For each $j \geq 0$ and taking $\Sigma_{-1} = \varnothing$, we have that each connected component of $\Sigma_{j} \setminus \Sigma_{j-1}$ satisfies one of the following conditions:
  \begin{itemize}
   \item If it is an orientable subsurface of genus $g$, $n$ punctures and $b$ boundary components, then $3g-3+n+b \geq 5$.
   \item If it is a non-orientable subsurface of genus $g$, $n$ punctures and $b$ boundary components, then $g + n + b \geq 8$.
  \end{itemize}
 \item Defining $B_{j}$ as the set of boundary curves of $\Sigma_{j}$, the set $B = \bigcup_{0 \leq j < \omega} B_{j}$ is a multicurve of $N$ composed of separating curves.
 \item Finally, $\displaystyle \bigcup_{0 \leq j < \omega} \Sigma_{j} = N$.
\end{enumerate}
\end{definition}

%%%%%%%%%%%%%%%%%%%%%%%%%%%%%%%%%%%%%%%%%%%%%%%%%%%%%%%%%%%%
%%%%%%%%%%%%%%%%%%%%%%%%%%%%%%%%%%%%%%%%%%%%%%%%%%%%%%%%%%%%
%%%%%%%%%%%%%%%%%%%%%%%%%%%%%%%%%%%%%%%%%%%%%%%%%%%%%%%%%%%%

\subsection{Ends}

In this subsection we recall the definition of ends and the classification of infinite type surfaces, for details we refer to \cite{Kerekjarto}, \cite{IRichards} and \cite{AramayonaVlamis}. First, we make a small note on the notation of the genus of a surface. By the classification of finite type surfaces, an orientable surface $S$ of genus $g$ is homeomorphic to the connected sum of $g$ tori minus $n$ points and $b$ open disks, or equivalently it is homeomorphic to a sphere with $n$ punctures, $b$ holes and $g$ handles. We say that each \emph{handle} adds one \emph{orientable genus} or \emph{positive genus}. If $N$ is a non-orientable surface of genus $g$ then it is homeomorphic to the connected sum of $g$ real projective planes minus $n$ points and $b$ open disks or equivalently homeomorphic to a sphere with $g$ crosscaps minus $n$ points and $b$ open disks. We say that each crosscap adds one \emph{non-orinetable genus} or one \emph{negative genus}. When we say infinite genus we think that an infinite number of handles or crosscaps have been added. Below there is a precise definition.

Let $N$ be an infinite type surface, an \emph{exiting sequence} is a sequence $\{U_i\}_{0 \leq i < \omega}$ of connected  open subsets of $N$ such that
\begin{itemize}
	\item $U_i\subset U_j$ whenever $j<i$;
	\item $U_i$ is not relatively compact for any $0 \leq i < \omega$;
	\item $U_i$ has compact boundary for all $0 \leq i < \omega$;
	\item any relatively compact subset of $N$ is disjoint from all but finitely many $U_i$'s. 
\end{itemize}
Let $\{U_i\}_{0 \leq i < \omega}$ be an exiting sequence, we say that an element $U_{i}$ is planar if it has genus zero. 

We say that two exiting  sequences are equivalent if every element of the first sequence is eventually contained in some element of the second, and vice versa. An equivalent class of an exiting sequence is called an \emph{end}, we denote the set of ends by $\End(N)$. The space of ends $\End(N)$ can be equipped with a topology that  makes it a totally disconnected, separable and compact set, then it is homeomorphic to a closed subset of the Cantor set $\calC$ (see Proposition 3 in \cite{IRichards}). We say that
\begin{itemize}
	\item An end $\{U_i\}_{0 \leq i < \omega}$ is \emph{orientable accumulated by genus} (or simply orientable) if  $U_i$ is orientable and has infinite orientable genus for all $i$.
	\item An end  $\{U_i\}_{0 \leq i < \omega}$ is \emph{non-orientable accumulated by genus} (or simply non-orientable) if  $U_i$ is non-orientable for all $i$.
	\item An end  $\{U_i\}_{0 \leq i < \omega}$ is \emph{planar} if $U_i$ is planar for all but finitely many $i$. 
\end{itemize} 

We denote by $\Endg(N)$ the subspace of orientable and non-orientable ends accumulated by genus, and by $\Endn(N)$ the subspace of non-orientable ends. The sets $\Endg(N)$ and $\Endn(N)$ are closed subsets of $\End(N)$.

We say that a surface $N$ is of \emph{infinite genus} (respectively \emph{infinitely non-orientable}) if there is no bounded subset $K\subset N$ such that $N-K$ is of genus zero (respectively orientable). Note that an infinitely non-orientable surface is also of infinite genus. It can happen that $N$ is of infinite genus and non-orientable but not infinitely non-orientable, in this case we say that $N$ is \emph{odd} or \emph{even} non-orientable according to whether every sufficiently large compact subsurface  is non-orientable of genus odd or even respectively (equivalently has an odd or an even number of crosscaps). With these definitions we have four \emph{orientability classes}: orientable, infinitely non-orientable, odd non-orientable and even non-orientable. 

Richards (see \cite{IRichards}) showed that the homeomorphism type of a surface is determinated by its genus, number of boundaries,  orientability class and the triple of spaces 
$$ \left(\End(N),\Endg(N),\Endn(N)\right).
$$

%%%%%%%%%%%%%%%%%%%%%%%%%%%%%%%%%%%%%%%%%%%%%%%%%%%%%%%%%%%%
%%%%%%%%%%%%%%%%%%%%%%%%%%%%%%%%%%%%%%%%%%%%%%%%%%%%%%%%%%%%
%%%%%%%%%%%%%%%%%%%%%%%%%%%%%%%%%%%%%%%%%%%%%%%%%%%%%%%%%%%%
%%%%%%%%%%%%%%%%%%%%%%%%%%%%%%%%%%%%%%%%%%%%%%%%%%%%%%%%%%%%
%%%%%%%%%%%%%%%%%%%%%%%%%%%%%%%%%%%%%%%%%%%%%%%%%%%%%%%%%%%%

\section{The Alexander method}\label{sec2}

In this section we prove the analogue of Theorem 1.1 in \cite{HMV1}, for non-orientable surfaces. To do this, we first recall what the Alexander method is for finite-type surfaces. Then we follow the general idea of the proof for the case of orientable infinite-type surfaces, highlighting the differences.

%%%%%%%%%%%%%%%%%%%%%%%%%%%%%%%%%%%%%%%%%%%%%%%%%%%%%%%%%%%%
%%%%%%%%%%%%%%%%%%%%%%%%%%%%%%%%%%%%%%%%%%%%%%%%%%%%%%%%%%%%
%%%%%%%%%%%%%%%%%%%%%%%%%%%%%%%%%%%%%%%%%%%%%%%%%%%%%%%%%%%%

\subsection{Finite type}\label{subsec2-1}

\begin{theorem}[The Alexander method for finite-type surfaces]\label{Thm-AM-finite}
 Let $\Sigma$ be a finite-type surface of genus $g$, $n$ punctures and $b$ boundary components. If $\Sigma$ is orientable, assume that $3g-3+n+b \geq 4$. If $\Sigma$ is non-orientable assume that $g + n + b \geq 5$. Then there exists a finite set of curves and arcs $\Gamma$, such that if $h \in \Homeo{\Sigma;\partial\Sigma}$ fixes the isotopy class of every element in $\Gamma$, then $h$ is isotopic to the identity.
\end{theorem}

\begin{figure}[htb]
	\centering
	\includegraphics[width=14cm]{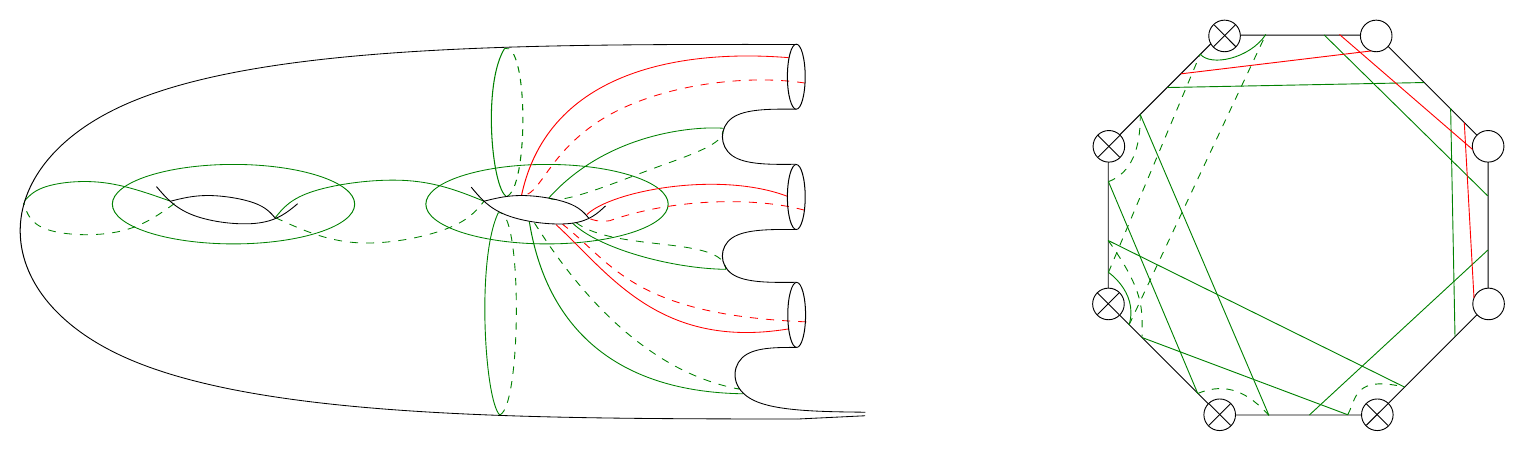}
	\caption{Two examples of a set of curves (in green) and arcs (in red) that determine homeomorphisms up to isotopy. On the left, an orientable surface (genus $2$, $3$ boundary components and $1$ puncture), while on the right a non-orientable surface (genus $5$ and $3$ boundary components).}
	\label{fig:Alexander-system}
\end{figure} 

Note that if $\Sigma$ has empty boundary, then $\Gamma$ does not contain arcs. Analogously, if $\Sigma$ has non-empty boundary, and we label the boundary components by $c_{1}, \ldots, c_{b}$, then we can obtain $\Gamma$ as follows: For each $i = 1, \ldots, b$, let $\alpha_{i}$ be an essential arc that starts and finishes at $c_{i}$, and also let $\Gamma^{\prime}$ be the set obtained from Theorem \ref{Thm-AM-finite} for $\mathrm{int}(\Sigma)$; then $\Gamma$ can be taken to be $\Gamma^{\prime} \cup \{\alpha_{i}\}_{i=1}^{b}$. See Figure \ref{fig:Alexander-system}.

%%%%%%%%%%%%%%%%%%%%%%%%%%%%%%%%%%%%%%%%%%%%%%%%%%%%%%%%%%%%
%%%%%%%%%%%%%%%%%%%%%%%%%%%%%%%%%%%%%%%%%%%%%%%%%%%%%%%%%%%%
%%%%%%%%%%%%%%%%%%%%%%%%%%%%%%%%%%%%%%%%%%%%%%%%%%%%%%%%%%%%

\subsection{Infinite type}\label{subsec2-2}

\begin{lemma}[cf. Lemma 3.2 in \cite{HMV1}]\label{Lemma-AM-Aux1}
 Let $k \geq 0$, and let $\Gamma_{1} = \{\alpha_{0}, \ldots, \alpha_{k}\}$ and $\Gamma_{2} = \{\beta_{0},\ldots,\beta_{k}\}$ be two collections of curves on $N$ that satisfy the following:
 \begin{enumerate}
  \item For each $i = 0, \ldots, k$, $\alpha_{i}$ is isotopic to $\beta_{i}$.
  \item For each $j= 1,2$, $\Gamma_{j}$ is a collection of pairwise disjoint curves in minimal position.
 \end{enumerate}
 Then, there exists $f \in \Homeo{N}$ isotopic to the identity, such that for all $i= 0, \ldots, k$, $f(\alpha_{i}) = \beta_{i}$. Moreover, $f$ can be chosen to be the end homeomorphism of an ambient isotopy of $N$.
\end{lemma}
\begin{proof}
 If $|\Gamma_{1}| = |\Gamma_{2}| = 1$, let $V_{0}$ be either an annular neighborhood or a Möbius band (depending on whether $\alpha_{0}$ and $\beta_{0}$ are two-sided curves or one-sided) such that $\alpha_{0}$ and $\beta_{0}$ are both contained in the interior of $V_{0}$ and they are isotopic to the central curve of $V_{0}$. Then, there exists an ambient isotopy $\widetilde{H}_{0}: V_{0} \times I \to V_{0}$ that deforms $\alpha_{0}$ into $\beta_{0}$, and that restricts to the identity on the boundary of $V_{0}$. We can then extend $\widetilde{H}_{0}$ to an ambient isotopy $H_{0}: N \times I \to N$ using the identity on $N \backslash V_{0}$; the homeomorphism $f := H_{0}(\cdot, 1)$ is the desired homeomorphism.
 
 We now proceed by induction. Suppose that there exists an ambient isotopy $H_{n}: N \times I \to N$ such that $f_{n} := H_{n}(\cdot,1)$ maps each $\alpha_{i}$ to $\beta_{i}$ for $i \leq n$. Note that the collection $f_{n}(\Gamma_{1}) = \{\beta_{0}, \ldots, \beta_{n}, f_{n}(\alpha_{n+1}), \cdots, f_{n}(\alpha_{k})\}$ is again a collection of pairwise disjoint curves in minimal position. Let $V_{n+1}$ be either an anullar neighborhood or a Möbius band such that $f_{n}(\alpha_{n+1})$ and $\beta_{n+1}$ are both contained in the interior of $V_{n+1}$ and they are isotopic to the central curve of $V_{n+1}$. We then obtain an ambient isotopy $\widetilde{H}_{n+1}: V_{n+1} \times I \to V_{n+1}$ as above. Extending $\widetilde{H}_{n+1}$ with the identity on $N \backslash V_{n+1}$ and doing an isotopy composition with $H_{n}$, we obtain an ambient isotopy $H_{n+1}: N \times I \to N$ that deforms $\alpha_{i}$ into $\beta_{i}$ for all $i \leq n+1$. We finish the proof by defining $f_{n+1}$ as $H_{n+1}(\cdot,1)$.
\end{proof}

Given a principal exhaustion of $N$, $\{N_{i}\}_{0 \leq i < \omega}$, the first step for the construction of $\Gamma$ is the following: For each $0 \leq i$, we define $B_{i}$ as the set of boundary curves of $N_{i}$. Also, we define $B = \bigcup_{0 \leq i < \omega} B_{i}$; we call $B$ the \textit{set boundaries of the principal exhaustion}.

\begin{lemma}[cf. Lemma 3.5 in \cite{HMV1}]\label{Lemma-AM-Aux2}
 Let $N_{0} \subset N_{1} \subset \cdots \subset N$ be a principal exhaustion of $N$, $B$ the set of boundaries of this principal exhaustion, and $h \in \Homeo{N}$ be such that for all $\beta \in B$, $h(\beta)$ is isotopic to $\beta$. Then there exists $g \in \Homeo{N}$ isotopic to $h$ such that $g|_{B} = \id|_{B}$.
\end{lemma}
\begin{proof}
 For this proof, we use Lemma \ref{Lemma-AM-Aux1} to define a sequence of homeomorphisms $g_{i}$ that satisfy the conclusion of the lemma for the set $B_{i}$ instead of $B$. Then we define the desired $g$ from this sequence.
 
 For $B_{0}$, by Lemma \ref{Lemma-AM-Aux1} there exists a homeomorphism $f_{0}: S \to S$ such that $f$ is isotopic to $\id$ and $f_{0}|_{B_{0}} = h|_{B_{0}}$. Then, we define $g_{0} \ColonEqq f_{0}^{-1} \circ h$; this implies that $g_{0}$ is isotopic to $h$ and $g_{0}|_{B_{0}} = \id|_{B_{0}}$. Note that $g_{0}(N_{0}) = N_{0}$ since $g_{0}$ and $\id$ coincide in the boundary of $N_{0}$.
 
 For $B_{1}$, we define $\widetilde{g_{0}} \ColonEqq g_{0}|_{N \backslash \mathrm{int}(N_{0})}$. Using Lemma \ref{Lemma-AM-Aux1} again, there exists $\widetilde{f_{1}}: N \backslash \mathrm{int}(N_{0}) \to N \backslash \mathrm{int}(N_{0})$ such that $\widetilde{f_{1}}$ is isotopic to $\id|_{N \backslash \mathrm{int}(N_{0})}$ and $\widetilde{f_{1}}|_{B_{1}} = g_{0}|_{B_{1}}$. We can then extend $\widetilde{f_{1}}$ to the whole surface using the identity on $N_{0}$, i.e. we define $f_{1}$ as follows: $$f_{1}(s) = \left\{ \begin{tabular}{cl} $s$ & $s \in N_{0}$\\ $\widetilde{f_{1}}(s)$ & otherwise \end{tabular}\right. ,$$ which in particular implies that $f_{1}$ is isotopic to $\id$ relative to $N_{0}$.
 
 Afterwards, we define $g_{1} \ColonEqq f_{1}^{-1} \circ g_{0}$. Thus, $g_{0}$ is isotopic to $g_{1}$ relative to $N_{0}$ and $g_{1}|_{B_{0} \cup B_{1}} = \id|_{B_{0} \cup B_{1}}$.
 
 Inductively, following the same procedure for the definition of $g_{1}$, we define for any $i \geq 2$ a homeomorphism $g_{i}: N \to N$ such that $g_{i}|_{\cup_{0 \leq k \leq i} B_{k}} = \id|_{\cup_{0 \leq k \leq i} B_{k}}$, and for $i < j$ we have that $g_{i}$ is isotopic to $g_{j}$ relative to $N_{i}$.
 
 Thus, the map $g: N \to N$ with $s \mapsto g_{i}(s)$ for $s \in N_{i}$ is a well-defined homeomorphism. Also, by construction, $g|_{B} = \id|_{B}$. Moreover, if $H_{i}:N \times [0,1] \to N$ is the isotopy from $g_{i}$ to $g_{i+1}$ relative to $N_{i}$, then let $\widetilde{H_{i}}$ be the rescaling of $H_{i}$ to the interval $[\frac{i}{i+1},\frac{i+1}{i+2}]$; thus, defining $H: N \times [0,1] \to N$ as the concatenation of the $\widetilde{H_{i}}$, and as $g = H|_{N \times \{1\}}$, we obtain an isotopy from $g_{1}$ to $g$. Therefore, by transitivity $g$ is isotopic to $h$, finishing the proof.
\end{proof}

\begin{proof}[\textbf{Proof of Theorem \ref{Thm-AlexMethod}}]
 Let $\{N_{i}\}_{0 \leq i < \omega}$ be a principal exhaustion of $N$, let $B$ be its boundaries. Let $\{\Sigma_{j}\}_{0 \leq j < \omega}$ be the collection of subsurfaces $\Sigma_{j}$ of $N$, corresponding to the connected components of $N \backslash B$. Note that for all $0 \leq j$, $\Sigma_{j}$ has complexity at least $5$ if it is orientable, and if $g$ is its genus and it has $n$ punctures, then $g + n \geq 8$ if it is non-orientable. Also, if we denote by $\overline{\Sigma_{j}}$ the closure of $\Sigma_{j}$ in $N$, then all the boundary curves of $\overline{\Sigma_{j}}$ are elements of $B$.
 
 Now, for each $\beta \in B$, let $\beta^{*}$ be a curve on $N$ such that:
 \begin{itemize}
  \item Their isotopy classes intersect at least twice, i.e. $i([\beta],[\beta^{*}]) \geq 2$.
  \item For all $\gamma \in B \backslash \{\beta\}$, $\beta^{*}$ is disjoint from $\gamma$.
 \end{itemize}
 
 Note that the choice of $\beta^{*}$ is arbitrary, and while for every $\beta \in B$, there  exist infinitely many possible choices for $\beta^{*}$, once the choice is made, we fix $\beta^{*}$ for the rest of the proof. We define $B^{*} = \{\beta^{*} : \beta \in B\}$. 
 
 Then, let $\Gamma_{j}$ be a finite set of curves in $\Sigma_{j}$ such that it satisfies the Alexander method for $\Sigma_{j}$ (see Theorem \ref{Thm-AM-finite}).
 
 We claim that the set: $$\Gamma = B \cup B^{*} \cup \left(\bigcup_{0 \leq j < \omega} \Gamma_{j}\right),$$ satisfies Theorem \ref{Thm-AlexMethod}.
 
 To prove this, let $h \in \Homeo{N}$ be such that $h(\gamma)$ is isotopic to $\gamma$ for all $\gamma \in \Gamma$. Due to Lemma \ref{Lemma-AM-Aux2}, we can suppose that $h|_{B} = \id|_{B}$. This implies that $h|_{\overline{\Sigma_{j}}} \in \Homeo{\overline{\Sigma_{j}};\partial\overline{\Sigma_{j}}}$ for each $0 \leq j < \omega$.
 
 By construction, for each $0 \leq j < \omega$ and every boundary curve $\beta$ of $\overline{\Sigma_{j}}$, $\beta^{*} \cap \Sigma_{j}$ contains at least one proper arc with endpoints in $\beta$. Thus, if we denote by $\Delta_{j}$ the set of boundary curves of $\Sigma_{j}$, we have that the set $\Gamma_{j} \cup \left(\bigcup_{\beta \in \Delta_{j}} \beta^{*} \cap \Sigma_{j}\right)$ satisfies the conditions of Theorem \ref{Thm-AM-finite}. Given that for each $0 \leq j < \omega$, $h|_{\overline{\Sigma_{j}}}$ preserves the isotopy class of this set of curves, we have that $h|_{\overline{\Sigma_{j}}}$ is isotopic to $\id|_{\overline{\Sigma_{j}}}$ relative to the boundary. Finally, we use these isotopies to define an isotopy from $h$ to $\id$.
\end{proof}

%%%%%%%%%%%%%%%%%%%%%%%%%%%%%%%%%%%%%%%%%%%%%%%%%%%%%%%%%%%%
%%%%%%%%%%%%%%%%%%%%%%%%%%%%%%%%%%%%%%%%%%%%%%%%%%%%%%%%%%%%
%%%%%%%%%%%%%%%%%%%%%%%%%%%%%%%%%%%%%%%%%%%%%%%%%%%%%%%%%%%%
%%%%%%%%%%%%%%%%%%%%%%%%%%%%%%%%%%%%%%%%%%%%%%%%%%%%%%%%%%%%
%%%%%%%%%%%%%%%%%%%%%%%%%%%%%%%%%%%%%%%%%%%%%%%%%%%%%%%%%%%%

\section{Isomorphisms between curve graphs}\label{sec3}

In this section we prove Theorem \ref{Thm-Atalan-Korkmaz}, which says that any isomorphism between curve graphs is induced by a homeomorphism between the underlying surfaces. Our proof of this theorem is almost the same as the proof of Theorem 1.1 in \cite{HMV2}, with the proofs being essentially the same (simply substituting auxiliary lemmata and results in the orientable case with the corresponding lemmata and results in the possibly non-orientable case); thus, while we refer the reader to \cite{HMV2} for more detailed proofs, for the sake of completeness we also sketch the proofs.

Throughout this section, let $N$, $N_{1}$ and $N_{2}$ be connected, possibly non-orientable surfaces of infinite type with empty boundary.\\[0.3cm]

Recalling from Section \ref{sec1} that in hyperbolic surfaces geodesic representatives of curves are always in minimal position, we know that for $N$ there exists a set $\mathcal{S}$ of representatives of the isotopy classes of all essential curves on $N$ such that any two elements of $\mathcal{S}$ are in minimal position. With this in mind we have the following lemma.

\begin{lemma}[cf. Lemma 2.5 in \cite{HMV2}]\label{Lemma-LocFinite}
 Let $N$ be an infinite-type surface and $\mathcal{S}$ be a set of representatives of the isotopy classes of all essential curves on $N$ such that any two elements of $\mathcal{S}$ are in minimal position. Let also $\Gamma$ be a set of isotopy classes of curves and $\mathcal{G}\subset \mathcal{S}$ be a set of representatives of $\Gamma$. Then the following are equivalent:
 \begin{enumerate}
  \item $\mathcal{G}$ is locally finite.
  \item $\Gamma$ is locally finite.
  \item For every curve $\alpha$ the set $\{\gamma \in \Gamma : i(\alpha, \gamma) \neq 0\}$ is finite.
 \end{enumerate}
\end{lemma}

Since this lemma is actually more general than the analogous in \cite{HMV2}, we give a more detailed proof.

\begin{proof}
 $(1) \Rightarrow (2)$: This is obvious by the definition of a set of isotopy classes being locally finite.
 
 $(2) \Rightarrow (3)$: Let $\mathcal{X}$ be a set of representatives of $\Gamma$ that is locally finite, $\alpha$ be a curve of $N$ and $a$ be a representative of $\alpha$. Then, we have the following: $$\{\gamma \in \Gamma : i(\alpha, \gamma) \neq 0\} \subset \{[c] \in \Gamma: c \in \mathcal{X}, a \cap c \neq \varnothing\}.$$ Since the latter set is finite, then we obtain (3).
 
 $(3) \Rightarrow (1)$: We prove this by contrapositive. Suppose $\mathcal{G}$ is not locally finite. Then there exists a compact set $K$ and an infinite collection $\{\gamma_{i}\}_{0 \leq i < \omega} \subset \mathcal{G}$ such that for all $i$ we have that $K \cap \gamma_{i} \neq \varnothing$. There exists a finite-type subsurface $\Sigma$ that satisfies the following:
 \begin{enumerate}
  \item $\Sigma$ contains $K$ in its interior.
  \item If $\{c_{1}, \ldots, c_{b}\}$ are all the boundary curves of $\Sigma$, then $\{c_{1}, \ldots, c_{b}\} \subset \mathcal{S}$.
 \end{enumerate}
 
 Then we have that for all $i$, $\Sigma \cap \gamma_{i} \neq \varnothing$, and we can divide the proof into two cases:
 
 \textbf{Case 1,} there exists an infinite subcollection $\{\gamma_{i_{n}}\}_{0 \leq n < \omega}$ contained in $\Sigma$: Let $P \subset \mathcal{S}$ be a pants decomposition of $\Sigma$. Since $P$ is a maximal set of pairwise disjoint and pairwise non-isotopic curves of $\Sigma$, by the pigeonhole principle, there is a curve $\alpha \in P$ that intersects infinitely many elements of $\{ \gamma_{i_{n}}\}_{0 \leq n < \omega}$. Given that all the elements in $\mathcal{S}$ are in minimal position, we have that the set $\{\gamma \in \Gamma : i(\alpha, \gamma) \neq 0\}$ is infinite.
 
 \textbf{Case 2,} only finitely many elements of $\{\gamma_{i}\}_{0 \leq i < \omega}$ are contained in $\Sigma$: If all (but finitely many of) the elements of $\{\gamma_{i}\}_{0 \leq i < \omega}$ were disjoint from all the elements of $\{c_{1}, \ldots, c_{b}\}$, then they would not intersect $K$. Thus by the pigeonhole principle there exists a boundary curve $c_{i}$ of $\Sigma$ and a subsequence $\{\gamma_{i_{n}}\}_{0 \leq n < \omega}$, such that they every $\gamma_{i_{n}}$ intersects $c_{i}$. Given that the elements of $\mathcal{S}$ are in minimal position, this implies that the set $\{\gamma \in \Gamma : i([c_{i}], \gamma) \neq 0\}$ is infinite.
\end{proof}

\begin{lemma}[cf. Corollary 2.6 in \cite{HMV2}]\label{Lemma-Isom-Multicurves}
 Let $N_{1}$, $N_{2}$ be two connected (possibly non-orientable) surfaces of infinite type, $M$ be a multicurve on $N_{1}$, and $\varphi: \ccomp{N_{1}} \to \ccomp{N_{2}}$ be an isomorphism.  Then, $\varphi(M)$ is a multicurve. In particular, if $P$ is a pants decomposition of $N_{1}$, then $\varphi(P)$ is a pants decomposition.
\end{lemma}
\begin{proof}[Sketch of the proof]
 Note that Lemma \ref{Lemma-LocFinite} implies that we can characterize local finiteness simplicially. As such, any multicurve can be characterized as a set satisfying certain simplicial conditions. Since $\varphi$ is an isomorphism these conditions are preserved, implying that its image is also a multicurve.
 
 Moreover, if $P$ is a pants decomposition, then it is a maximal multicurve. Then, by the argument above, $\varphi(P)$ is a multicurve, and maximality is obtained by the surjectivity of $\varphi$.
\end{proof}

We say a set $M$ of locally finite pairwise disjoint curves bounds a closed subsurface $\Sigma \subset N$, if the set of boundary curves of $\Sigma$ that are not boundary curves of $N$, is exactly $M$.

Recall that a \textit{pair of pants} is a closed subsurface whose interior is homeomorphic to a thrice-punctured sphere. Now, let $P$ be a pants decomposition of $N$ and $\alpha_{1}, \alpha_{2} \in P$ be different. For each $i = 1,2$, let $\beta_{i}$ be $\alpha_{i}$ if $\alpha_{i}$ is two-sided; otherwise, let $\beta_{i}$ be the boundary curve of the Möbius band that is the regular neighborhood of $\alpha_{i}$ (note that in this case, $\beta_{i}$ is not essential). We say that $\alpha_{1}$ and $\alpha_{2}$ are \textit{adjacent \wrt}$P$, if there exists a set $M \supset \{\beta_{1},\beta_{2}\}$ that bounds a pair of pants.

\begin{remark}\label{Remark-Adjacency}
 Note that $\alpha, \beta \in P$ are adjacent \wrt$P$ if and only if there exists a curve $\gamma$ such that $i(\alpha,\gamma) \neq 0 \neq i(\beta,\gamma)$ and $i(\delta,\gamma) = 0$ for all $\delta \in P \backslash \{\alpha,\beta\}$. This implies that we can simplicially characterise adjacency.
\end{remark}

\begin{lemma}\label{Lemma-Isom-Adjacency}
 Let $N_{1}$, $N_{2}$ be two connected (possibly non-orientable) surface of infinite type, $P$ be a pants decomposition on $N_{1}$, and $\varphi: \ccomp{N_{1}} \to \ccomp{N_{2}}$ be an isomorphism. Then, $\alpha,\beta \in P$ are adjacent \wrt $P$ if and only if $\varphi(\alpha)$ and $\varphi(\beta)$ are adjacent \wrt $\varphi(P)$.
\end{lemma}
\begin{proof}
 This follows immediatly from Remark \ref{Remark-Adjacency} and the fact that $\varphi$ is an isomorphism.
\end{proof}

Let $P$ be a pants decomposition; we define its \textit{adjacency graph}, denoted as $\adj{P}$, as the simplicial graph whose vertex set is $P$, and two vertices span an edge if they are adjacent \wrt $P$. Note that by Lemma \ref{Lemma-Isom-Multicurves}, if $\varphi: \ccomp{N_{1}} \to \ccomp{N_{2}}$ is an isomorphism, then $\varphi(P)$ is a pants decomposition; thus, we induce a map $\varphi_{P}: \adj{P} \to \adj{\varphi(P)}$ defined as $\alpha \mapsto \varphi(\alpha)$. Then, the following corollary follows from Lemmata \ref{Lemma-Isom-Multicurves} and \ref{Lemma-Isom-Adjacency}.

\begin{corollary}[cf. Proposition 3.1 in \cite{HMV2}]\label{Cor-Isom-AdjGraph}
 Let $N_{1}$, $N_{2}$ be two connected (possibly non-orientable) surface of infinite type, $P$ be a pants decomposition on $N_{1}$, and $\varphi: \ccomp{N_{1}} \to \ccomp{N_{2}}$ be an isomorphism. Then, $\varphi_{P}$ is an isomorphism between $\adj{P}$ and $\adj{\varphi(P)}$.
\end{corollary}

A separating curve is called \textit{outer} if it bounds a twice-punctured disk, while it is called non-outer otherwise.

\begin{remark}\label{Remark-Non-Outer}
 Note that if $P$ is a pants decomposition, then non-outer separating curves are exactly the cut vertices of $\adj{P}$. Moreover, if $M \subset P$ is a set of non-outer separating curves, then $M$ bounds a finite-type closed subsurface of $N$ if and only if there is a finite subgraph of $\adj{P}$ delimited exactly by the vertices corresponding to $M$ in $\adj{P}$.
\end{remark}

\begin{lemma}[cf. Lemma 3.2 in \cite{HMV2}]\label{Lemma-Isom-Separating}
 Let $N_{1}$, $N_{2}$ be two connected (possibly non-orientable) surface of infinite type, and $\varphi: \ccomp{N_{1}} \to \ccomp{N_{2}}$ be an isomorphism. Then, $\alpha$ is a non-outer separating curve if and only if $\varphi(\alpha)$ is a non-outer separating curve.
\end{lemma}
\begin{proof}
 This follows from Remark \ref{Remark-Non-Outer} and the fact that $\varphi$ is an isomorphism.
\end{proof}

Now, to prove Theorem \ref{Thm-Atalan-Korkmaz} we use the following theorem, which is an amalgamation of Theorem 2 in \cite{Atalan-Korkmaz}, Theorem 1 in \cite{Ivanov}, Theorem 1 in \cite{Korkmaz}, Theorem (a) in \cite{Luo}, and Theorem 1 in \cite{Atalan-Korkmaz}.

\begin{theorem}\label{Thm-Finite-type-Atalan-Korkmaz}
 Let $N_{1}$ and $N_{2}$ be two finite-type surfaces such that none of them are homeomorphic to any of the following surfaces: $S_{0,4}$, $S_{1,1}$, $S_{0,5}$, $S_{1,2}$, $S_{0,6}$, $S_{2,0}$. If $\phi: \ccomp{N_{1}} \to \ccomp{N_{2}}$ is an isomorphism, then $N_{1}$ and $N_{2}$ are homeomorphic and $\phi$ is induced by a homeomorphism $N_{1} \to N_{2}$.
\end{theorem}
\begin{proof}
 See the articles cited above.
\end{proof}

\begin{proof}[\textit{\textbf{Proof of Theorem \ref{Thm-Atalan-Korkmaz}}}]
 Let $\Sigma_{0} \subset \Sigma_{1} \subset \cdots N_{1}$ be a principal exhaustion of $N_{1}$, $B$ be the set o boundaries of this principal exhaustion, and for each $0 \leq i < \omega$ let $B_{i}$ the boundary curves of $\Sigma_{i}$.
 
 Since for each $0 \leq i < \omega$, $B_{i}$ is a set of non-outer separating curves, then $\varphi(B_{i})$ is composed solely of non-outer separating curves.
 
 Then, let $P \supset B$ be a pants decomposition of $N_{1}$; for each $i \in \mathbb{Z}^{+}$, note the following two facts:
 \begin{enumerate} 
  \item The set of curves from $P \backslash B$ contained in $\Sigma_{i}$ forms a pants decomposition of $\Sigma_{i}$. We denote this pants decomposition as $P_{i}$.
  \item By Remark \ref{Remark-Non-Outer} there is a finite subgraph of $\adj{P}$ delimited exactly by the cut vertices corresponding to $B_{i}$, and this finite subgraph corresponds to $P_{i}$.
 \end{enumerate}
 
 Given that for each $0 \leq i < \omega$, the set $\varphi(B_{i})$ is composed solely by non-outer separating curves, and using Lemma \ref{Lemma-Isom-Multicurves}, Corollary \ref{Cor-Isom-AdjGraph} and point (2) above, there exists (for each $0 \leq i < \omega$) a finite subgraph $\varphi(P_{i})$ in $\adj{\varphi(P)}$ delimited exactly by the cut vertices corresponding to $\varphi(B_{i})$. Again for each $0 \leq i < \omega$, by Remark \ref{Remark-Non-Outer}, there exists a finite-type closed subsurface $\Sigma_{i}^{\prime}$ bounded by $\varphi(B_{i})$.
 
 Recalling that for any $i = 1,2$ and any finite-type subsurface $\Sigma \subset N_{i}$, there is a natural embedding $\ccomp{\Sigma} \hookrightarrow \ccomp{N_{i}}$ induced by the inclusion, we denote the image of this inclusion in $\ccomp{N_{i}}$ also as $\ccomp{\Sigma}$. With this, for any $\alpha \in \mathcal{V}(\ccomp{\Sigma_{i}})$, either $\alpha \in P_{i}$ or there exists $\beta \in P_{i}$ such that $i(\alpha,\beta) \neq 0$. Since this is preserved by $\varphi$, we can induce a map $\varphi_{i}: \ccomp{\Sigma_{i}} \to \ccomp{\Sigma_{i}^{\prime}}$, defined as $\alpha \mapsto \varphi(\alpha)$; given that $\varphi$ is an isomorphism, we have that $\varphi_{i}$ is also an isomorphism for each $0 \leq i < \omega$.
 
 Using Theorem \ref{Thm-Finite-type-Atalan-Korkmaz}, we have that (for each $0 \leq i < \omega$) $\Sigma_{i}$ is homeomorphic to $\Sigma_{i}^{\prime}$ and there exists a homeomorphism $f_{i}: \mathrm{int}(\Sigma_{i}) \to \mathrm{int}(\Sigma_{i}^{\prime})$ such that $\varphi_{i} = f_{i}$. Given that for any $i < j$ we have that $\varphi_{i} = \varphi_{j}|_{\ccomp{N_{i}}}$, by the Alexander method we obtain that $f_{i} = f_{j}|_{\mathrm{int}(N_{i})}$. With this, we can define a map $f: N_{1} \to N_{2}$ as $x \mapsto f_{i}(x)$ if $x \in \mathrm{int}(\Sigma_{i})$. By construction this map is a homeomorphism that induces $\varphi$.
\end{proof}

%%%%%%%%%%%%%%%%%%%%%%%%%%%%%%%%%%%%%%%%%%%%%%%%%%%%%%%%%%%%
%%%%%%%%%%%%%%%%%%%%%%%%%%%%%%%%%%%%%%%%%%%%%%%%%%%%%%%%%%%%
%%%%%%%%%%%%%%%%%%%%%%%%%%%%%%%%%%%%%%%%%%%%%%%%%%%%%%%%%%%%
%%%%%%%%%%%%%%%%%%%%%%%%%%%%%%%%%%%%%%%%%%%%%%%%%%%%%%%%%%%%
%%%%%%%%%%%%%%%%%%%%%%%%%%%%%%%%%%%%%%%%%%%%%%%%%%%%%%%%%%%%

\section{Topological generation of $\PMod{N}$}

The group $\Homeo{N}$ with the compact-open topology is a topological group; thus, $\Mod{N}$ inherits a very natural topological group structure with the quotient of the compact-open topology. In this work, we abuse language and refer to this topology of $\Mod{N}$ as ``compact-open topology'' too.

In the finite-type surface case, it not hard to see that $\Mod{N}$ becomes a discrete group with this topology (the Alexander Method is a simple way of seeing this). However, in the infinite-type case this does not happen; moreover, $\Mod{N}$ is not even locally compact (see \cite{AramayonaVlamis}). As such, $\Mod{N}$ (and its subgroups) becomes much more interesting from a topological-group viewpoint. In this section, we obtain several results considering $\Mod{N}$ with this topology; all these results have analogues in the orientable-surface case, and their proofs here are essentially the same but with subtle differences (see \cite{PatelVlamis} and \cite{AramayonaPatelVlamis}).

%%%%%%%%%%%%%%%%%%%%%%%%%%%%%%%%%%%%%%%%%%%%%%%%%%%%%%%%%%%%
%%%%%%%%%%%%%%%%%%%%%%%%%%%%%%%%%%%%%%%%%%%%%%%%%%%%%%%%%%%%
%%%%%%%%%%%%%%%%%%%%%%%%%%%%%%%%%%%%%%%%%%%%%%%%%%%%%%%%%%%%

\subsection{$\Mod{N}$ is Polish}

Recall that a \emph{Polish group} is a topological group whose underlying space is Polish (a separable and completely metrizable space). We verify that $\Mod{N}$ is a Polish group.

Let $\Gamma$ be a graph with a countable set of vertices. We define a topology on $\Aut{\Gamma}$ as follows: for any finite vertex subset $A$ of the vertex set of $\Gamma$ define
$$ 
U_A = \{ f\in \Aut{\Gamma} : f(a)=a \text{ for all } a\in A\}.
$$
Then, the \emph{permutation topology} on $\Aut{\Gamma}$ is defined as the topology with basis the translated sets $f\cdot U_A$, where $A$ is a finite set of vertex of $\Gamma$ and  $f\in \Aut{\Gamma}$.  With this topology $\Aut{\Gamma}$ is separable and even more it is a Polish group (see Lemma 2.2 in \cite{AramayonaPatelVlamis}).

The curve graph $\ccomp{N}$ of an infinite type surface $N$ (orientable or not) has a countable set of vertex. A simple way of seeing this is to consider a principal exhaustion; each curve graph of these subsurfaces has a countable set of vertices and the union of all these sets of vertices is equal to the set of vertices of $\ccomp{N}$. Hence, we can equip $\Aut{\ccomp{N}}$ with the permutation topology. From Corollary \ref{Cor-Atalan-Korkmaz} we have that $\Mod{N} \cong \Aut{\ccomp{N}}$, and using the isomorphism $\Psi$ we can pullback the permutation topology to $\Mod{N}$. We call this topology the \emph{permutation topology} in $\Mod{N}$. Recall that this topology has as basis the $\Mod{N}$-translates of the sets 
$$
U_A = \{ f\in \Mod{N} : f(a)=a \text{ for all } a\in A\},
$$
where $A$ is any finite set of the vertices of $\ccomp{N}$. Then, we have that $\Mod{N}$ is a Polish group with the permutation topology. 

Using the Alexander method for finite surfaces (see Theorem \ref{Thm-AM-finite}) it can be proved that the compact-open topology is the same as the permutation topology in $\Mod{N}$, the proof for orientable surfaces works for the non-orientable ones too. For the sake of completeness we include a sketch here. 

First we make the following observation: Let $K$ and $U$ respectively be a compact set and an open set of $N$. Note that the set $[K,U] \ColonEqq \{f \in \Homeo{N}: f(K) \subset U\}$ is a sub-basic open neighborhood of $\Homeo{N}$ with the compact-open topology. As such, the set $\bnm{K,U} \ColonEqq \{f \in \Mod{N}: \exists F \in f, F(K) \subset U\}$ is a sub-basic open neighborhood of $\Mod{N}$.

\begin{lemma}\label{lemma:open-comp:permutation:same}
Let $N$ be a connected (possibly non-orientable) surface of infinite topological type. Then, the compact-open topology and the permutation topology in $\Mod{N}$ coincide.
\end{lemma}

\begin{proof}[Sketch of proof.]
Let $\tau$ be the compact-open topology and $\tau'$ be the permutation topology. Let $U_A\in\tau'$ be a basic open of the identity $id$, and suppose $A = \{\alpha_{1}, \ldots, \alpha_{n}\}$. Consider for each $\alpha_i\in A$ the  basic open $V_i = \bnm{\alpha_i,N(\alpha_i)}$ where $N(\alpha_i)$ is a regular neighborhood of $\gamma_i$. Let $V = V_1\cap\cdots\cap V_n$, the set $V$ is a basic open of $id$ $\tau$, and $V = U_A$. Then $\tau'\subset \tau$.

Conversely, let $U\in \tau$ be a sub-basic open of $id$, we have that $U=\bnm{K,V}$ with $K$ and $V$ compact and open subsets of $N$ respectively. Let $\Sigma$  be a connected, compact subsurface of $N$ such that $K \subset \Sigma$. Let $A$ be a finite set of curves such that $A\cap\Sigma$ is collection of arcs and curves in $\Sigma$ that satisfies the Alexander method for finite-type surfaces (see Theorem \ref{Thm-AM-finite}). From the Alexander method we have that each $f\in U_A$ is isotopic to the identity in $\Sigma$. In particular, there exists $F\in f$ such that $F(K)=K\subset V$, hence $U_A \subset U$.

Therefore $\tau =\tau'$.
\end{proof}

Corollary \ref{Cor-Top-Group-Iso} now follows from lemma \ref{lemma:open-comp:permutation:same}. To finish this subsection we make a small note on convergence in $\Mod{N}$. 

Given $\Sigma$ a compact subsurface of $N$ and $f\in \Mod{N}$, the set $\{\, g \in \Mod{N} \mid \, g|_{\Sigma} = f|_{\Sigma} \,\}$ is open. Let $f \in \Mod{N}$ and $(f_{i})_{0 \leq i < \omega}$ be a sequence in $\Mod{N}$. By definition of the topology, the sequence $(f_{i})_{0 \leq i < \omega}$ converges to $f$ if and only if for all open neighborhoods $V$ of $f$ there exists $N \geq 0$ such that $f_{i} \in V$ for all $i \geq N$. Then, if $\Sigma_{1} \subset \Sigma_{2} \subset \cdots N$ is a principal exhaustion and $(f_{i})_{0 \leq i < \omega}$ converges to $f$, for all $i$ there exists an $n_i$ such that for all $j\geq n_i$ we have that $f_j |_{\Sigma_i} = f|_{\Sigma_i}$. 

%%%%%%%%%%%%%%%%%%%%%%%%%%%%%%%%%%%%%%%%%%%%%%%%%%%%%%%%%%%%
%%% DEFINICIÓN handle-shift %%%%%%%%%%%%%%%%%%%%%%%%%%%%%%%%
%%%%%%%%%%%%%%%%%%%%%%%%%%%%%%%%%%%%%%%%%%%%%%%%%%%%%%%%%%%%

\subsection{Handle-shifts}\label{section:handleshifts}

The concept of a handle-shift was introduced in  \cite{PatelVlamis}, here we recall their definition. Let $\oshiftsup$ the surface obtained by taking $\R \times [-1,1]$, removing the interior of each disk of radius $\frac{1}{4}$ with center in $(n,0)$ for each $n \in \Z$, and attaching a torus with one boundary component to each such a disk. If instead  of a torus we attach projective planes with one boundary component, we obtain a surface denoted by $\nshiftsup$. See Figure \ref{fig:oSigmanSigma}.

\begin{figure}[htb]
	\centering
	\includegraphics[scale = 1]{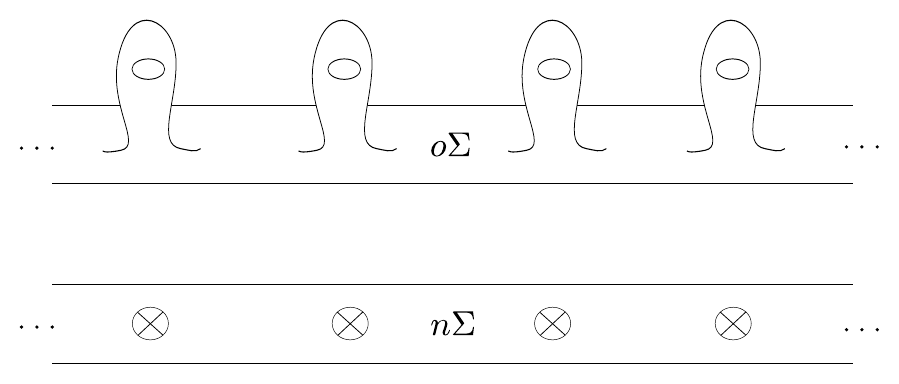}
	\caption{The surface $\oshiftsup$ above and the surface $\nshiftsup$ below.}
	\label{fig:oSigmanSigma}
\end{figure} 

Let $\fun{\sigma}{\oshiftsup}{\oshiftsup}$ be the homeomorphism up to isotopy determined by requiring 
\begin{enumerate}
	\item $\sigma(x,y) = (x+1,y)\,$ for $\,(x,y)\in \R\times [-1+\epsilon, 1- \epsilon]\,$ for some $\,\epsilon > 0$ and
	\item $\sigma(x,y) = (x,y)\,$ for $\,(x,y)\in \R\times \{-1,1\}$.
\end{enumerate}
Similarly, let $\fun{n\sigma}{\nshiftsup}{\nshiftsup}$ be the homeomorphism determined by similar conditions to that of $\sigma$.

\bigskip
Consider a surface $N$ with at least two non planar ends and let $\fun{h}{N}{N}$ be a homemorphism. If there is an embedding $\fun{\iota}{\oshiftsup}{N}$ inducing an injection $\fun{\iota_{*}}{\End(\oshiftsup)}{\End(N)}$ such that
$$
h(x) = 
\begin{cases}
(\iota \circ \sigma \circ \iota^{-1})(x) &  \mbox{if} \quad x\in \iota(\oshiftsup),\\
x  &  \mbox{otherwise,}
\end{cases}
$$
we say that $h$ is:
\begin{enumerate}
	\item an \textit{orientable handle-shift} if both elements of $\iota_{*}(\End(\oshiftsup))$ are orientable,
	\item a \textit{semi-orientable handle-shift} if exactly one element of $\iota_{*}(\End(\oshiftsup))$ is orientable,
	\item a \textit{pseudo-orientable handle-shift} if both elements of $\iota_{*}(\End(\oshiftsup))$ are not orientable.
\end{enumerate}

Finally we say that a homeomorphism $\fun{h}{N}{N}$ is a \textit{non-orientable handle-shift} if there exists an embedding $\fun{\iota}{\nshiftsup}{N}$ inducing an injection $\fun{\iota_{*}}{\End(\nshiftsup)}{\End(N)}$ and such that 
$$
h(x) = 
\begin{cases}
(\iota \circ n\sigma \circ \iota^{-1})(x) &  \mbox{if} \quad x\in \iota (\nshiftsup),\\
x  &  \mbox{otherwise.}
\end{cases}
$$

We also call a handle-shift to the mapping class associated to a handle-shift.  Notice that the power of a handle-shift is not a handle-shift, this can be proved using the Alexander method.  Also notice that in all cases the homeomorphism $h$ has an attracting and a repelling end denoted by $h^{+}$ and $h^{-}$ respectively, and they are invariant under isotopies (see \cite{AramayonaPatelVlamis}).

Now, let the \textit{ends graph of} $N$, denoted by $EG(N)$, be the simplicial graph whose vertex set is $\Endg(N)$, and such that two vertices $x,y$ span an edge if there exists a handle-shift $h$ with $\{x,y\} = \{h^{+},h^{-}\}$. Note that with the discourse above it is clear that $EG(N)$ is a complete connected graph. 

More recently, P. Patel and C.R. Abbott in \cite{AbbottPatel} have defined ``generalised handle-shifts'', as homeomorphisms constructed as above, with the difference that they consider $\oshiftsup$ to be possibly obtained by gluing to the infinite band $\mathbb{R} \times [-1,1]$ other surfaces besides a torus/projective plane with one boundary component.

%%%%%%%%%%%%%%%%%%%%%%%%%%%%%%%%%%%%%%%%%%%%%%%%%%%%%%%%%%%%
%%%%%%%%%%%%%%%%%%%%%%%%%%%%%%%%%%%%%%%%%%%%%%%%%%%%%%%%%%%%
%%%%%%%%%%%%%%%%%%%%%%%%%%%%%%%%%%%%%%%%%%%%%%%%%%%%%%%%%%%%
\subsection{$\PMod{N}$ as a closed subgroup}

For the following theorem, recall that a curve $\alpha$ is called one-sided if its closed regular neighborhood is homeomorphic to a Möbius band, and it is called two-sided it if its closed regular neighborhood is an annulus.

\begin{theorem}\label{Thm-TopGen-PMod}
 Let $N$ be a connected (possibly non-orientable) surface of infinite topological type, and let $H$ be the set of all handle-shifts of $N$. Then, $\PMod{N} = \overline{\langle \PModc{N} \cup H \rangle}$.
\end{theorem}
\begin{proof}The proof of this theorem is essentially the same  that Patel-Vlamis have for orientable surfaces in  Proposition $6.2$ of \cite{PatelVlamis}. We just need to take in account the following observations:

\begin{enumerate}
\item A pants decomposition of a non-orientable surface $N$ may contain one-sided curves, and one-sided curves bound exactly one pair of pants (instead of two pair of pants as two-sided curves do).
\item One-sided curves are not separating curves. Also, if $\alpha$ and $\gamma$ are two one side curves on a finite type surface $S$, there exist $f\in \PMod S$ such that $f(\alpha)=\gamma$.
\item Some steps of Patel-Vlamis proof involve comparing two subsurfaces  of $N$, say $V$ and $W$. They do it by comparing their genus. In the general case it is better to see whether $V$ and $W$ are homeomorphic or not.
\item In the last paragraph of Patel-Vlamis proof, for the general case, it is a good idea to consider the subsurfaces $V$ and $W$  as  connected sums of a torus plus zero, one or two projective planes. In this way it becomes more obvious which composition of handle-shifts is the mapping class $h$. 
\end{enumerate}
\end{proof}

%%%%%%%%%%%%%%%%%%%%%%%%%%%%%%%%%%%%%%%%%%%%%%%%%%%%%%%%%%%%
%%%%%%%%%%%%%%%%%%%%%%%%%%%%%%%%%%%%%%%%%%%%%%%%%%%%%%%%%%%%
%%%%%%%%%%%%%%%%%%%%%%%%%%%%%%%%%%%%%%%%%%%%%%%%%%%%%%%%%%%%

\subsection{Definition of $\{h_{i}\}_{0 \leq i < r}$}\label{subsec:definehi}

Using Theorem \ref{Thm-TopGen-PMod}, we have that if $N$ has at most one end accumulated by genus, then $\PMod{N} = \overline{\PModc{N}}$, proving the first part of Theorem \ref{Thm-PMod-TopGen-hi}. Thus, for the rest of this section we assume that $N$ has at least two ends accumulated by genus, unless otherwise stated.

The purpose of this subsection is to construct a collection $\{h_i\}_{0 \leq i < r}$ of handle-shifts such that  $\PMod{N} = \overline{\langle \PModc{N} \cup \{h_{i}\}_{0 \leq i < r} \rangle}$ (see Theorem  \ref{Thm-PMod-TopGen-hi}). This construction is inspired by the one made in the proof of Theorem 5 in \cite{AramayonaPatelVlamis}: Given a surface $\Sigma$ with either empty boundary or compact boundary, we define $\widehat{\Sigma}$ as the surface obtained by forgetting the planar ends and capping the boundary components with disks. The idea is to find a collection of separating simple closed curves $\{\gamma_i\}_{0 \leq i < r}$ such that their homology classes generate $H^{sep}_{1}(\widehat{N} ; \Z)$  and assign to each one of this curves a handle-shift. 

However, unlike the orientable case treated in \cite{AramayonaPatelVlamis}, if $N$ is a surface with at least one orientable end and two non-orientable ends we cannot take just any such a collection of curves; the problem is that it is not clear how to build non-orientable handle-shifts using only $\PModcc{N}$, pseudo-orientable, semi-orientable and orientable handle-shifts, so we need to choose a \emph{good basis} of $H^{sep}_{1}(\widehat{N} ; \Z)$ in order to have in $\{h_{i}\}_{0 \leq i < r}$ the minimum number of non-orientable handle-shifts needed to generate all the non-orientable handle-shifts.

\subsubsection{A fixed model for $N$} \label{sec:model:N}

We start by fixing a  model of a given infinite-type surface $N$. Recall the following theorem of I. Richards.

\begin{theorem}[cf. Theorem 2 in \cite{IRichards}]\label{thm:construccion:superficie}
	Given a triple $\left(X,Y,Z \right)$ of  compact, separable, totally disconnected spaces $Z\subset Y\subset X$, there is a surface $\Sigma$ whose ends $\left(\End(\Sigma), \Endg(\Sigma), \Endn(\Sigma) \right)$ are topologically equivalent to the triple $(X,Y,Z)$.
\end{theorem}

In the proof of Theorem \ref{thm:construccion:superficie}, Richards gives an explicit construction of the surface $\Sigma$. So, given an infinite-type surface $N$  and using Richards' construction, we obtain a surface $\Sigma$ homeomorphic to $N$, by abuse of notation we denote it also by $N$. We recall Richards' construction here.

Recall that $\End(N)$ is homeomorphic to a subset of the Cantor set $\calC$. Embed the cantor set $\calC$ in the one point compactification of the plane as the set of all points $(x,0)$ such that $0\leq x\leq 1$ and $x$ admits a triadic expansion which does not involve the digit $1$. Let $\calD'$ be the collection of all closed disks in the plane whose diameters are the intervals in the $x$ axis 
$$\left[\frac{n-\frac{1}{3}}{3^m}, \frac{n+\frac{4}{3}}{3^m} \right], 
$$
with $m,n \in \Z$ such that $m\geq 1$, $\, 0\leq n \leq 3^m$ and $n$ admits a triadic expansion free from $1$'s. Let $\calD$ be the subcollection consisting of all disks in $\calD '$ which contain at least one point of $\End(N)$. For each disk $\kappa \in \calD$, let $\kappa_1$ and $\kappa_2$ the two largest disks in $\calD'$ properly contained in $\kappa$. Choose two circles $C^+(\kappa)$ and $C^-(\kappa)$ contained  in the interior of $\kappa$ such that:
\begin{enumerate}
	\item $C^+(\kappa)$ is contained in the upper half-plane and $C^-(\kappa)$ is contained in the lower half-plane.
	\item $C^+(\kappa)$ and $C^-(\kappa)$ do not intersect $\kappa_1$ and $\kappa_2$.
	\item $C^+(\kappa)$ and $C^-(\kappa)$ are symmetric with respect to the $x$ axis.
\end{enumerate}
Remove the interior of $C^{\pm}(\kappa)$ for all $\kappa \in \calD$ such that $\kappa \cap \Endg(N) \neq \emptyset$. If $\kappa \cap \Endn(N) = \emptyset$, then identify the boundaries of $C^+(\kappa)$ and $C^-(\kappa)$ by reflecting $C^+(\kappa)$ in the $x$ axis preserving orientation, i.e. add a handle. If $\kappa \cap \Endn(N) \neq \emptyset$, then identify $C^+(\kappa)$ and $C^-(\kappa)$ by translating $C^+(\kappa)$ onto $C^-(\kappa)$, i. e. add a Klein bottle.

If the surface $N$ is of finite genus orientable or non-orientable, in the above construction, add the finite number of handles or crosscaps that are needed. Similarly if $N$ is of infinite genus, but odd or even non-orientable add one or two crosscaps respectively.

%%%%%%%%%%%%%%%%%%%%%%%%%%%%%%%%%%%%%%%%%%%%%%%%%%%%%%%%%%%%
%%%%%%%%%%%%%%%%%%%%%%%%%%%%%%%%%%%%%%%%%%%%%%%%%%%%%%%%%%%%
%%%%%%%%%%%%%%%%%%%%%%%%%%%%%%%%%%%%%%%%%%%%%%%%%%%%%%%%%%%%

\subsubsection{A family of curves $\calH$}\label{sec:fam:curves:H}

Let $N$ be an arbitrary infinite-type surface with empty boundary.  With the model of a surface described in Subsubsection \ref{sec:model:N},  we construct a family of pairwise disjoint and non-homologous separating simple closed curves $\calH = \{\gamma_i\}_{0 \leq i < r}$, with $1 \leq r \leq \omega$ and such that each connected component of the surface $N^\prime$ obtained from $N$ by removing pairwise disjoint regular neighborhoods of each $\gamma \in \calH$, has exactly one end accumulated by genus (orientable or non-orientable).

Recall that $H_{1}^{sep}(N ; \Z)$ denotes the subgroup of $H_{1}(N;\Z)$ that is generated by homology classes that can be represented by separating simple closed curves on the surface. Notice that Lemma $4.2$ of \cite{AramayonaPatelVlamis} is also valid for non-orientable surfaces, for the sake of completeness we enounce here this lemma. 

\begin{lemma}\label{Lemma-LimHsep}
	If $\{N_{i}\}_{0 \leq i < \omega}$ is a principal exhaustion of $N$, then 
	$$
	H_{1}^{sep}(N ; \Z) \, = \, \lim_{i \rightarrow \infty}H_{1}^{sep}(N_{i} ; \Z). 
	$$
	In particular, there exists $0 \leq n,m < \omega$ such that every non-zero element $\nu \in H_{1}^{sep}(N ; \Z)$ can be written as 
	$$\nu = \sum_{i=1}^{m} a_{i}\nu_{i} $$     
	where $a_{i}\in \Z $ and $\nu_{i}$ can be represented by a peripheral curve on $N_{i}$.
\end{lemma}

It follows from Lemma \ref{Lemma-LimHsep} and the fact that $H_{1}^{sep}(\widehat{N};\Z)$ is a free abelian group, that there exists a collection of pairwise disjoint and non-homologous separating simple closed curves $\{\gamma_{i}\}_{0 \leq i <r}$ on $\widehat{N}$ such that 
$$
H_{1}^{sep}(\widehat{N} ; \Z) \,=\, \bigoplus_{0\leq i < r} \, \langle \nu_{i}\rangle 
$$  
where $\nu_{i}$ denotes the homology class of $\gamma_{i}$. Notice we can choose the curves $\gamma_i$ such that they do not intersect the boundaries and planar ends of $N$. Also notice that $r$ could be $\omega$, i.e. $H_{1}^{sep}(\widehat{N} ; \Z)$ could be infinitely generated.

If the surface $N$ is one of the following types
\begin{itemize}
	\item orientable,
	\item even or odd non-orientable,
	\item non-orientable and $\Endg(N)= \Endn(N)$,
\end{itemize}
we define $\calH$ as the family of pairwise disjoint and non-homologous separating simple closed curves such that their homology classes generate $H_{1}^{sep}(\widehat{N} ; \Z)$ (see above). If $N$ is not one of these cases, then $\Endn(N)\neq \emptyset$ and $\Endg(N) \neq \Endn(N)$. We define $\calH$ for this type of surfaces in the following paragraphs.

Suppose the set of ends of $N$ satisfies that $\Endn(N)\neq \emptyset$ and $\Endg(N) \neq \Endn(N)$. Consider a model for $N$ as the one described in Subsubsection \ref{sec:model:N} and for simplicity suppose $\widehat{N}=N$. We choose any non-orientable end $e_{-}^0 \in \Endn(N)$. Let $\kappa_0\in \calD$ be such that $e_{-}^{0}\in \kappa_0$ and there exist at least one more $\kappa \in \calD$ with the same diameter. Let $\kappa_1,\dots \kappa_{r_{1}} \in \calD$ be the disks different to $\kappa_0$ but  with the same diameter and denote by $\gamma_{1,1}^{0},\dots,\gamma_{1,r_{1}}^{0}$ their respective boundaries. Define
$$
\calH_{\gamma,1}^{0}= \{\gamma_{1,1}^{0},\dots,\gamma_{1,r_{1}}^{0}\}.
$$ 

By removing from $N$, pairwise disjoint regular neighborhoods of each $\gamma_{1,i}^{0} \in \calH_{\gamma,1}^{0}$, we obtain a surface $N^{0}$ with $r_{1}+1$ connected components, each of one has infinite genus and satisfies one of the following:
\begin{enumerate}[label= \alph*)]
	\item has only orientable ends,
	\item has only non-orientable ends,
	\item has both orientable and non-orientable ends.  
\end{enumerate}

Denote each connected component of $N^{0}$  by $N_{1,j}^{0}$ with $1\leq j \leq r_{1}+1$ and where the first $0 \leq t_{0}\leq r_{1}+1 $ satisfies c) and the other $r_{1}+ 1 -t_{0}$ connected components satisfy either a) or b).

If $t_0 < r_1 + 1$, for each subsurface $N_{1,j}^{0}$ with $t_{0} < j \leq r_{1} +1$  let $\calH_{1,j}^{0}$ be a family of  pairwise disjoint simple closed curves such that their homology classes generate $H_{1}^{sep}(\widehat{N}_{1,j}^0 ;\Z)$. The curves in $\calH_{1,j}^{0}$ can be chosen such that they are boundary curves of elements of $\calD$. Define
$$\calH^{0}= \calH_{\gamma,1}^{0} \cup  \left( \bigcup_{i= t_{0} +1}^{r_{1}+1}  \calH_{1,i}^{0} \right) 
$$

Let $n_1 = t_0$. If $n_1 = 0$ we are done and we define $\calH = \calH^0$. If $n_1 \geq 1$, for  $1\leq j \leq n_1$ 
rename each surface $N_{1,j}^0$ as $ N_{j}^1$. Repeating the algorithm, we obtain the following for each  $N_{j}^1$:  
\begin{enumerate}[label= \roman*)]
	\item A set 
	$$\calH_{\gamma,j}^{1} = \{ \gamma_{j,1}^{1},\dots, \gamma_{j,r_j}^{1} \}
	$$
	of separating simple closed curves which are boundaries of disks on $\calD$.
	\item A family of subsurfaces $N_{j,i}^{1}$ with $1 \leq  i \leq r_j +1$, and where the first $0 \leq t_{j}\leq r_j +1$ satisfy c) and the rest $r_{j}+1-t_{j}$ satisfy either a) or b).
	\item If $t_j < r_j + 1$, for each $N_{j,i}^{1}$ with $t_{j} < i \leq r_{j} +1$, a set $\calH_{j,i}^{1}$ of pairwise disjoint simple closed curves which are boundaries of disks on $\calD$ and their homology classes generate $H_{1}^{sep}(\widehat{N}_{j,i}^{1};\Z)$.   
\end{enumerate}

For each $1 \leq j \leq n_1$, if $t_j < r_j + 1$ let  
$$
\calH_{j}^1 \, = \, \calH_{\gamma,j}^1 \cup \left( \bigcup_{i=t_j +1}^{r_j +1} \calH _{j,i}^{1}\right)
$$
and if $t_j = r_j + 1 $, let $\calH_{j}^{1} = \calH_{\gamma,j}^{1}$. Define 
$$
\calH^1 \, = \, \calH^0 \cup \bigcup_{j=0}^{n_{1}} \calH_{j}^1.
$$

Let $n_2 = \sum_{j=1}^{n_1}t_j$. If $n_2 = 0$ we are done an let $\calH = \calH^{1}$. If $n_2 > 0$, rename all the subsurfaces $N_{j,i}^{1}$ with $1 \leq j \leq n_{1}$ and $1\leq i \leq t_{j}$  as $N_{j}^{2}$, with $1\leq j \leq n_{2}$. We can repeat the previous procedure to obtain a set of curves $\calH^2$. Continuing in this way we define
$$
\calH \, = \, \bigcup _{0 \leq i \leq n} \calH^i
$$
with $0 \leq n \leq \omega$; notice that since this algorithm may not finish in finite time, we have to allow the possibility of $n = \omega$. If $N\neq \widehat{N}$ we define $\calH$ as set of curves obtained by applying the above algorithm to $\widehat{N}$. 

%%%%%%%%%%%%%%%%%%%%%%%%%%%%%%%%%%%%%%%%%%%%%%%%%%%%%%%%%%%%
%%%%%%%%%%%%%%%%%%%%%%%%%%%%%%%%%%%%%%%%%%%%%%%%%%%%%%%%%%%%
%%%%%%%%%%%%%%%%%%%%%%%%%%%%%%%%%%%%%%%%%%%%%%%%%%%%%%%%%%%%

\subsubsection{Construction of $\{h_{i}\}_{0 \leq i < r}$}\label{subsubsection:constructionhi}

\medskip

In Subsubsection \ref{sec:fam:curves:H}, for any infinite-type surface $N$ we obtained a family of curves $\calH$, in this subsection we associate a handle-shift to each $\gamma \in \calH$. The association is similar to the one made in the proof of Theorem 5 in \cite{AramayonaPatelVlamis}, for the sake of completeness we include it here. 

First denote by  $N^\prime$ the surface obtained by removing from $\widehat{N}$ pairwise disjoint regular neighborhoods of each $\gamma \in \calH$. We have the following observations: 
\begin{enumerate}[label= \roman*)]
	\item Each closed curve $\gamma \in \calH$ is the boundary of a disk $\kappa_{\gamma} \in \calD$.
	\item The homology classes of the closed curves of $\calH$ are linearly independent and generate $H_{1}^{sep}(\widehat{N}; \Z)$.  
	\item Each separating curve in $N^\prime$ bounds a compact surface.
	\item  Each connected component of $N^{\prime}$  has one end accumulated by genus.
	\item If $N$ has at least two non-orientable ends, let $N_{1}$ be a connected component of $N^{\prime}$ with non-orientable end. Then, there exist a connected component $N_2$ of $N^{\prime}$, different to $N_1$,  and  a curve $\gamma \in \calH$ such that: $N_2$ has a non-orientable end and, $N_1$ and $N_2$ have a boundary isotopic to $\gamma$ in $N$. This happens by construction of the familly $\calH$.
\end{enumerate}

For simplicity, suppose $N = \widehat{N}$ and put a sub-index to each curve of $\calH$, in other words $\calH = \{\gamma_{i} \}_{0\leq i < r}$ where $1 \leq r \leq \omega$. Recall that every connected component of $N^\prime$ is classified up to homeomorphism by its orientability class and the number of boundary components (see \cite{IRichards}).

As is described in \cite{AramayonaPatelVlamis}, let $Y$ be the surface obtained from $[0,1]\times [1,\infty) \subset \R^{2}$ by attaching periodically infinitely many tori as in the construction of $\oshiftsup$. Similarly, let $Y'$ be the surface obtained from $[0,1] \times [-1,-\infty) \subset \R^{2}$ by attaching periodically infinitely many projective planes.  

On the other hand, consider $\R^{2}$ and remove $n$ open disks centered along the horizontal axis. Attaching an infinite number of tori periodically and vertically above each removed disk we obtain the one-ended infinite-genus orientable surface with $n$ boundary components that in \cite{AramayonaPatelVlamis} is denoted by $Z_{n}$. If below each removed disk we also attach, periodically and vertically, an infinite number of projective planes we obtain the one ended infinitely non-orientable surface with $n$ boundary components $W_n$. If instead of attaching an infinite number of projective planes to $Z_n$ we attach an even (respectively odd) number we obtain the one-ended infinite-genus even non-orientable (respectively odd non-orientable) surface with $n$ boundary components that we denote by $eZ_{n}$ (respectively $oZ_{n}$).

Observe that there are $n$ disjoint embeddings of $Y$ into $Z_n$, $W_n$, $eZ_n$ and $oZ_n$, respectively such that in each one the image of $[0,1]\times \{1\} \subset Y$ is contained in a unique boundary component of the surface under consideration  and each boundary component contains only one of such images. In the surface $W_n$ we also have $n$ disjoint embeddings of $Y'$ satisfying similar conditions. 

Fix $\gamma_{i} \in \{\gamma_{i} \}_{0\leq i < r}$ and let $N_{1}$ and $N_{2}$ be the two connected components of $N'$ that have a boundary component homotopic to $\gamma_{i}$ in $N$. Denote this boundaries by $b_1 \subset N_1 $ and $b_2 \subset N_2$. We have the following three cases depending on the orientation of the ends of $N_1$ and $N_2$.

\begin{enumerate}
	\item If the end of $N_1$ and $N_2$ are both orientable, as is done in \cite{AramayonaPatelVlamis}, let $Y_{j}$ denote the image of $Y$ in $N_{j}$ intersecting $b_{j}$, for $j=1,2$. The intervals $Y_{1}\cap b_1$ and $Y_2 \cap b_2$ can be connected with a strip $T \cong [0,1]\times [0,1]$ in the regular neighborhood of $\gamma_{i}$. The surface $\oshiftsup$ can be embedded in $N$ with image $\oshiftsup_{i} = Y_1 \cup T \cup Y_2$. We have two orientable handle-shifts supported in $\oshiftsup_i$, choose one and denote it by $h_i$.
	\item If the end of $N_1$ is orientable and the end of $N_2$ is non-orientable (or vice versa), we also have an embedding of $\oshiftsup$ with image $\oshiftsup_{i} = Y_1 \cup T \cup Y_2$ but now we have two semi-orientable handle-shifts supported in $\oshiftsup_{i}$, again choose one and denoted it by $h_{i}$.
	\item If the end of $N_1$ and $N_2$ are both non-orientable, denote by $Y_{j}'$ the image of $Y'$ in $N_{j}$ intersecting $b_{j}$, $j=1,2$. Proceeding as in the previous cases we have now an embedding of $\nshiftsup$ in $N$ with image $\nshiftsup_{i} = Y_{1}' \cup T \cup Y_{2}'$ and in consequence we have two non-orientable handle-shifts supported in $\nshiftsup_{i}$, choose one and denote it by $h_{i}$.
\end{enumerate}

To each $\gamma_{i} \in \{ \gamma_{i}\}_{0 \leq i < r}$ we assign it the corresponding handle-shift $h_{i}$. Notice that:

\begin{enumerate}
	\item The support of $h_{i}$ intersect $\gamma_{j}$ if and only if $i = j$, hence $\langle h_{i}, h_{j} \rangle$ is a free abelian group, and its rank is 2 if and only if $i \neq j$.
	\item In the case $3$, there is also two pseudo-orientable handle-shifts, however we do not choose any one of them because it is not clear that doing that we can generate all $\PMod{N}$.
	\item For each $h_{i}$, the ends $h_i^{+}$ and $h_{i}^{-}$ span an edge of $EG(N)$ which by abuse of notation we also denote by $h_{i}$. The set of all this ends and edges form a maximal tree of $EG(N)$ denoted by $TEG(N)$. Even more, we give an arbitrary orientation to $EG(N)$ with the condition that every edge $h_{i}$ in $TEG(N)$ has initial vertex $\iota(h_{i}) = h^{-}_{i}$ and terminal vertex $\tau(h_{i}) = h^{+}_{i}$.
	\item The set of vertices of $EG(N)$ and edges $h_i$ that correspond to non-orientable ends and non-orientable handle-shifts respectively, form a subtree of $TEG(N)$ that we denote by $nTEG(N)$. 
\end{enumerate} 

Consider the abelian subgroup topologically generated by the handle-shifts  $\{h_{i}\}_{0\leq i < r}$ with the subgroup topology. It is not difficult to prove that this group is homeomorphic to the group $\Z^{r}$ with the product topology.

\subsection{Proof of Theorem \ref{Thm-PMod-TopGen-hi}}

\begin{figure}[htb]
	\centering
	\includegraphics[scale = 1]{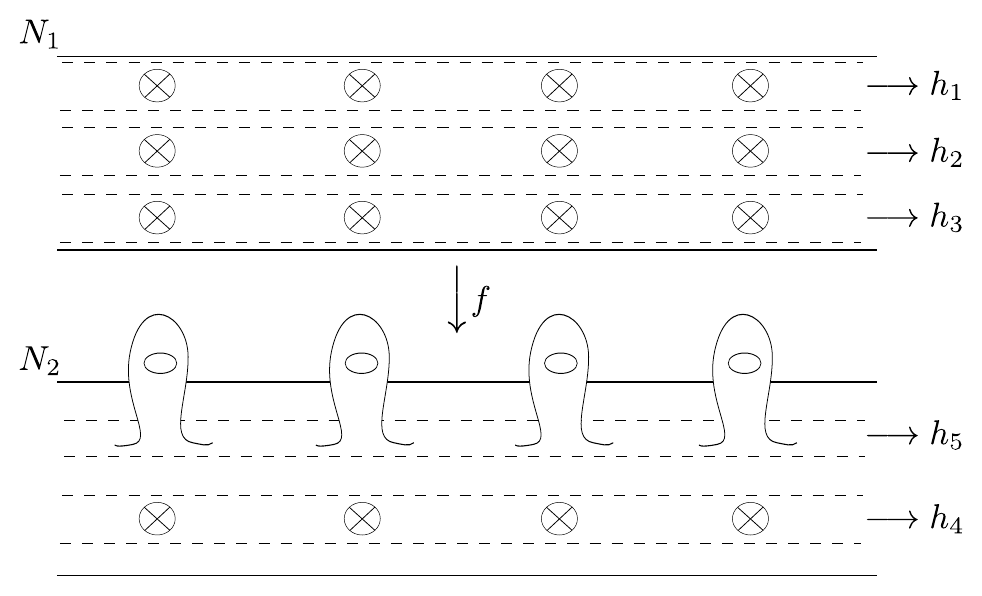}
	\caption{$ [ h_1h_2h_3] = f^{-1}[h_5h_4]f.$}
	\label{fig:5:handle:relation}
\end{figure} 

 Before we start the proof of Theorem \ref{Thm-PMod-TopGen-hi}, consider the homeomorphic surfaces $N_1$ and $N_2$ that are shown in Figure \ref{fig:5:handle:relation}. Let $f$ be the homeomorphism that sends a neighborhood of each column of three crosscaps on $N_1$ to a neighborhound of a handle and a crosscap in $N_2$ (the ones that are in the correponding column in Figure \ref{fig:5:handle:relation}). It is not hard to convince yourself that 
\begin{equation} \label{5:hshift:relation}
[ h_1 \circ h_2 \circ h_3] = f^{-1}[h_5\circ h_4]f.
\end{equation}
Notice that in $N_{1}$, $[f^{-1}\circ h_5  \circ f]$ is a pseudo-orientable handle-shift and  $[f^{-1}\circ h_4 \circ f]$ is a non-orientable handle-shift.

\medskip

By Theorem \ref{Thm-TopGen-PMod} $\overline{\langle \PModc{N} \cup \{h_i\}_{0 \leq i < r}   \rangle} \subseteq \overline{\langle \PModc{N} \cup H   \rangle}$, then to prove Theorem \ref{Thm-PMod-TopGen-hi} it is enough to prove that $H$ is in  $\overline{\langle \PModc{N} \cup \{h_i\}_{0 \leq i < r} \rangle}$. Given $h\in H$, the idea is to use $\PModcc{N}$ and $\{h_i\}_{0 \leq i < r}$ to build a handle-shift $\overline{h}$ with the same ends that $h$, then using $\PModcc{N}$ modify such $\overline{h}$ to get $h$. We do this in four cases.
\begin{description}
	\item[Non-orientable handle-shifts:] Suppose $h\in H$ is non-orientable. The attracting and repelling ends $h^+$ and $h^-$ are non-orientable. Then, in $nTEG(N)$ there exists a path $\gamma$ that goes from $h^-$ to $h^+$. Using the handle-shifts corresponding to edges of $\gamma$ and $\PModcc{N}$, we can construct a non-orientable handle-shift $\overline{h}$ such that $\overline{h}^+=h^+$ and $\overline{h}^- = h^-$. Finally, using $\PModcc{N}$ we can modify $\overline{h}$ to obtain $h$. Therefore $h\in \overline{\langle \PModc{N} \cup \{h_i\}_{0 \leq i <r}   \rangle}$.
	\item[Pseudo-orientable handle-shifts:] Suppose $h\in H$ is pseudo-orientable. The attracting and repelling ends $h^+$ and $h^-$ are non-orientable. We can build $h$ by using non-orientable handle-shifts with ends $h^+$ and $h^-$, $\PModcc{N}$ and relations of the type (\ref{5:hshift:relation}) above. As, by the previous case, all non-orientable handle-shifts are in $\overline{\langle \PModc{N} \cup \{h_i\}_{0\leq i <r}   \rangle}$, we have that $h\in \overline{\langle \PModc{N} \cup \{h_i\}_{0\leq i <r}   \rangle}$. 
	\item[Semi-orientable handle-shifts:] Suppose $h\in H$ is semi-orientable. Without loss of generality we can suppose that $h^+$ is orientable and $h^-$ is non-orientable. In $TEG(N)$ there exists a path $\gamma$ from $h^-$ to $h^+$ that first consists of non-orientable handle-shifts in $\{h_i\}_{0 \leq i < r}$, followed by one semi-orientable handle-shift $h_\gamma\in \{h_i\}_{0\leq i< r}$ and finally followed by orientable handle-shifts in $\{h_i\}_{0\leq i< r}$. Suppose that $h_{\gamma}^-$ is non-orientable (if not take $h_{\gamma}^{-1}$). Using $\PModcc{N}$, the orientable handle-shifts of $\gamma$, $h_{\gamma}$ and pseudo-orientable handle-shifts with ends $h^-$	and $h_{\gamma}^-$ we can construct a semi-orientable handle-shift $\overline{h}$	such that $\overline{h}^+=h^+$ and $\overline{h}^- = h^-$. Finally, using $\PModcc{N}$ we can modify $\overline{h}$ to get $h$. Therefore $h\in \overline{\langle \PModc{N} \cup \{h_i\}_{0 \leq i < r}   \rangle}$.
	\item[Orientable handle-shifts:] Suppose $h\in H$ is orientable. Let $\gamma$ be a path in $TEG(N)$ from $h^-$ to $h^+$. If all the vertices of $\gamma$ are orientable then all the edges correspond to orientable handle-shifts. By using these handle-shifts and $\PModcc{N}$ we can construct an orientable handle-shift $\overline{h}$ such that $\overline{h}^+=h^+$ and $\overline{h}^- = h^-$. By using $\PModcc{N}$ we can modify $\overline{h}$ to get $h$. If $\gamma$ has non-orientable vertices, then it has two semi-orientable handle-shift $h_{\gamma1}$ and $h_{\gamma2}$. Using $h_{\gamma1}$, $h_{\gamma2}$, pseudo-orientable handle-shifts and $\PModcc{N}$ we can build a handle-shift $\overline{h}$ that can be modified, using $\overline{\PModc{N}}$, to get h. Therefore $h\in \overline{\langle \PModc{N} \cup \{h_i\}_{0 \leq i < r}   \rangle}$.
\end{description}	

\medskip
\medskip
%%%%%%%%%%%%%%%%%%%%%%%%%%%%%%%%%%%%%%%%%%%%%%%%%%%%%%%%%%%%
%%%%%%%%%%%%%%%%%%%%%%%%%%%%%%%%%%%%%%%%%%%%%%%%%%%%%%%%%%%%
%%%%%%%%%%%%%%%%%%%%%%%%%%%%%%%%%%%%%%%%%%%%%%%%%%%%%%%%%%%%
%%%%%%%%%%%%%%%%%%%%%%%%%%%%%%%%%%%%%%%%%%%%%%%%%%%%%%%%%%%%
%%%%%%%%%%%%%%%%%%%%%%%%%%%%%%%%%%%%%%%%%%%%%%%%%%%%%%%%%%%%

\section{Proof Theorem \ref{Thm-Semi-direct-prod} and Corollary \ref{Thm-First-cohom-grp}}
This section is dedicated to the proof of Theorem \ref{Thm-Semi-direct-prod} and Corollary \ref{Thm-First-cohom-grp}. For Theorem \ref{Thm-Semi-direct-prod}, we use the collection of handle-shifts $\{h_{i}\}_{0 \leq i <r}$ (constructed in the previous section) to define a group homomorphism
$$
\overline{\varphi} :\overline{\langle \PModc{N} \cup \{h_{i}\}_{0 \leq i < r}   \rangle} \longrightarrow  \Z^{r}.
%\overline{\varphi} :\overline{\langle \PModc{N}, \{h_{i}\}_{i=1}^{r}   \rangle} \longrightarrow  \prod_{i=1}^{r} \, \langle h_{i} \rangle.
$$    
The proof of  Theorem \ref{Thm-Semi-direct-prod} is a corollary from the fact that $\overline{\varphi}$ induces a short-exact sequence that splits.

Afterwards, we use the same argument used in \cite{AramayonaPatelVlamis} to prove Corollary \ref{Thm-First-cohom-grp}.

%%%%%%%%%%%%%%%%%%%%%%%%%%%%%%%%%%%%%%%%%%%%%%%%%%%%%%%%%%%%
%%%%%%%%%%%%%%%%%%%%%%%%%%%%%%%%%%%%%%%%%%%%%%%%%%%%%%%%%%%%
%%%%%%%%%%%%%%%%%%%%%%%%%%%%%%%%%%%%%%%%%%%%%%%%%%%%%%%%%%%%

\subsection{ The homomorphism $\overline{\varphi}$. }
\medskip

In the following we continue using the same set of curves $\{\gamma_{i}\}_{0\leq i<r}$ and handle-shifts $\{h_{i}\}_{0\leq i<r}$ that were defined in the previous section.

Let $F(\PModc{N}\cup\{h_{i}\}_{0\leq i<r})$ be the free group generated by $\PModc{N}\cup\{h_{i}\}_{0\leq i<r}$, and let $e_{i} \in \Z^{r}$ be the sequence with $1$ in the $i$-th coordinate and $0$ everywhere else. We define a homomorphism $$\widetilde{\varphi}: F(\PModc{N} \cup \{h_{i}\}_{0\leq i<r}) \to \Z^{r},$$ defining $\widetilde{\varphi}(f) = \vec{0}$ for all $f \in \PModc{N}$, $\widetilde{\varphi}(h_{i}) = e_{i}$ for all $0 \leq i < r$, and extending via the universal property of free groups.

The next lemma proves that $\widetilde{\varphi}$ induces a homomorphism 
$$
\varphi : \langle \PModc{N} \cup \{h_{i}\}_{0\leq i<r} \rangle \longrightarrow \Z^{r}.
$$  
\begin{lemma}
 Let $\pi: F(\PModc{N} \cup \{h_{i}\}_{0\leq i<r}) \to \langle \PModc{N}, \{h_{i}\}_{0\leq i<r} \rangle$ be the canonical projection, and $w$ be a reduced word in $F(\PModc{N} \cup \{h_{i}\}_{0\leq i<r})$. If $\pi(w) = \id$, then $\widetilde{\varphi}(w) = \vec{0}$.
\end{lemma}
\begin{proof}
 Without loss of generality, we can assume that $w = w_{0}h_{i_{1}}^{\varepsilon_{1}}w_{1} \cdots w_{n}h_{i_{n}}^{\varepsilon_{n}}w_{n+1}$ with $w_{i} \in \PModc{N}$ and $\varepsilon_{i} \pm 1$.
 
 Since $w$ has finite length, there exists a compact surface $\Sigma$ such that $\pi(w)|_{N \setminus \Sigma} = (h_{i_{1}}^{\varepsilon_{1}} \circ \cdots \circ h_{i_{n}}^{\varepsilon_{n}})|_{N \setminus \Sigma}$. And given that $\pi(w) = \id$, we have then that $(h_{i_{1}}^{\varepsilon_{1}} \circ \cdots \circ h_{i_{n}}^{\varepsilon_{n}})|_{N \setminus \Sigma} = \id|_{N \setminus \Sigma}$. The elements of $\{h_{i}\}_{0\leq i<r}$ have pairwise disjoint support, hence if for some $i$ there exists $i_{m}$ such that $h_{i} = h_{i_{m}}$, then there exists some $i_{k}$ such that $h_{i} = h_{i_{k}}$ and $\varepsilon_{m} = -\varepsilon_{k}$.
 
 Thus we have the following: $$\widetilde{\varphi}(w) = \widetilde{\varphi}(w_{0}) + (\varepsilon_{1})\widetilde{\varphi}(h_{i_{1}}) + \cdots + (\varepsilon_{n})\widetilde{\varphi}(h_{i_{n}}) + \widetilde{\varphi}(w_{n}) = (\varepsilon_{1})\widetilde{\varphi}(h_{i_{1}}) + \cdots + (\varepsilon_{n})\widetilde{\varphi}(h_{i_{n}}) = \vec{0}.$$
\end{proof}

Due to the previous lemma, we have that $\widetilde{\varphi}$ descends to a homomorphism: $$\varphi: \langle \PModc{N} \cup \{h_{i}\}_{0 \leq i <r} \rangle \to \Z^{r}.$$

Now, for each $0 \leq i < r$, let $\psi_{i} := \pi_{i} \circ \varphi$, where $\pi_{i}$ is the canonical projection to the $i$-th coordinate. This is obviously a homomorphism, but before we prove it is continuous we need an auxiliary lemma and a definition.

\begin{lemma}\label{lemma:hisubstitutedPmodc}
 Let $0 \leq i <r$ be fixed, and $A$ be a finite set of curves such that $A$ does not separate the ends of $N$ corresponding to the ends of $\supp(h_{i})$. Then for any $g \in \PMod{N}$, there exists $f \in \PModc{N}$ such that $h_{i}|_{g(A)} = f|_{g(A)}$.
\end{lemma}
\begin{proof}
 Since $A$ does not separate the ends of $N$ corresponding to the ends of $\supp(h_{i})$ and $g \in \PMod{N}$, we have that $g(A)$ does not separate them too. If $g(A)$ is disjoint from $\supp(h_{i})$, then $f = \id$. So, suppose that $g(A)$ does intersect $\supp(h_{i})$.
 
 Let $K$ be a compact subset in $\overline{\supp(h_{i})}$ that contains the intersection of $g(A)$ with $\supp(h_{i})$ and has exactly one boundary component, and let $\Sigma_{1}$ and $\Sigma_{2}$ be two subsurfaces of $\supp(h_{i})$ homeomorphic to either a Möbius strip or a torus with a boundary component (depending if $h_{i}$ is non-orientable or not), such that $\Sigma_{1}$ and $\Sigma_{2}$ are both disjoint from $K$. Finally, let $a$ be an arc with endpoints in the boundary components of $\Sigma_{1}$ and $\Sigma_{2}$, such that $a$ is disjoint from $g(A)$. See Figure \ref{fig:2:hisubstitution} for an example.
 
 \begin{figure}[htb]
	\centering
	\includegraphics[width=15cm]{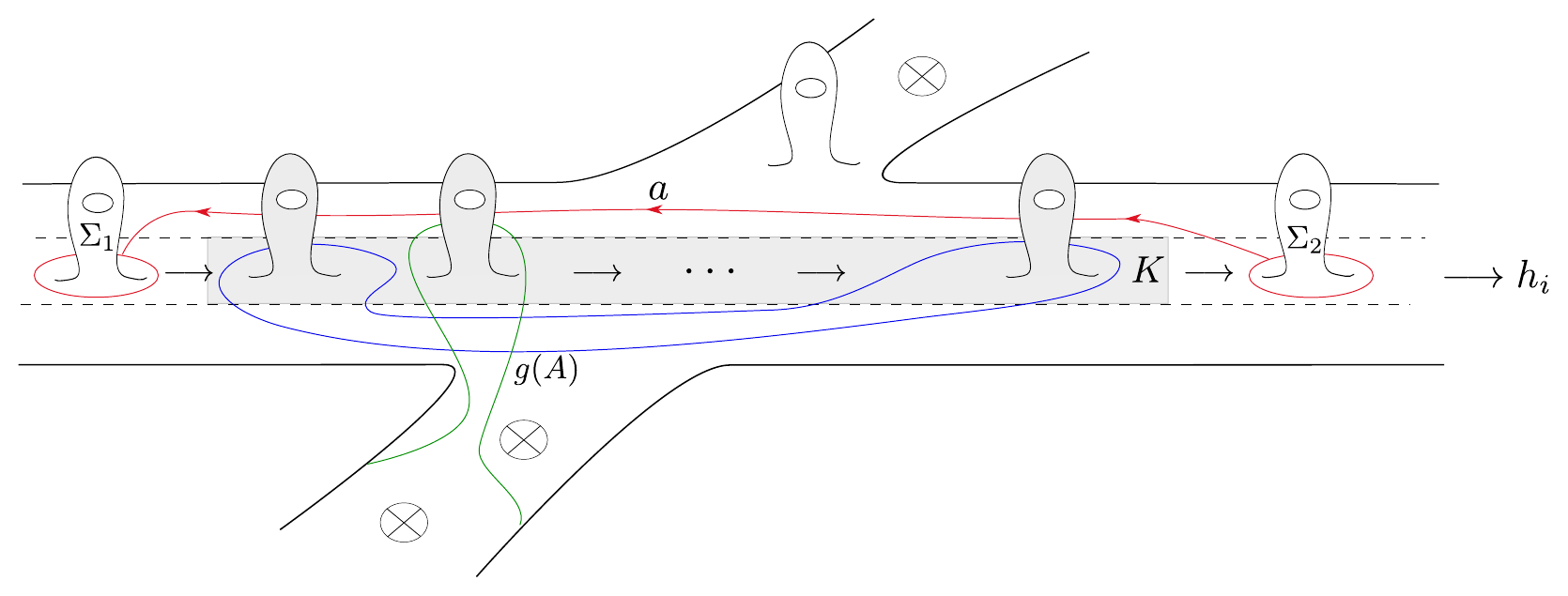}	\caption{How to substitute $h_{i}$ with an element of $\PModc{N}$.}
	\label{fig:2:hisubstitution}
 \end{figure} 
 
 We define $f$ as follows: slide $\Sigma_{2}$ through $a$ while shifting $\Sigma_{1}$ once in the direction that had originally $\Sigma_{2}$. See Figure \ref{fig:2:hisubstitution}. Thus, $f \in \PModc{N}$ and $h_{i}|_{g(A)} = f|_{g(A)}$ as desired.
\end{proof}

Recall that given a separating curve $\alpha$ in a finite-type surface $\Sigma$, the \textit{genus of} $\alpha$ is the minimum of the genus of the connected components of $\Sigma \setminus \alpha$. 

\begin{lemma}\label{lemma:psiicontinuous}
 For each $0 \leq i <r$, the homomorphism $\psi_{i}: \langle \PModc{N} \cup \{h_{j}\}_{0\leq j <r} \rangle \to \Z$ is continuous.
\end{lemma}
\begin{proof}
 To prove that $\psi_{i}$ is continuous, it suffices to prove that $\ker(\psi_{i})$ is an open set, and for that purpose we prove that the open set $V = \bnm{\gamma_{i},N(\gamma_{i})} \cap \langle \PModc{N} \cup \{h_{j}\}_{0\leq j <r} \rangle$ (here $\gamma_{i}$ is the curve used to define $h_{i}$ in Subsubsection \ref{subsubsection:constructionhi} and $N(\gamma_{i})$ is a regular neighborhood of $\gamma_{i}$) is a neighborhood of the identity contained in $\ker(\psi_{i})$.
 
 Let $f \in V$. If $f \in \PModc{N}$ or $f = h_{j}$ with $j \neq i$, it is obvious that $\psi_{i}(f) = 0$.
 
 Now, let $f$ be of the form $f = w_{0} \circ h_{i_{1}}^{\varepsilon_{1}} \circ \cdots \circ h_{i_{n}}^{\varepsilon_{n}} \circ w_{n}$ with $\varepsilon_{j} = \pm 1$ for all $j$. Given that $\gamma_{i}$ does not separates the ends of the support of $h_{j}$ for $j \neq i$, by a repeated use of Lemma \ref{lemma:hisubstitutedPmodc} we can define $\widetilde{f}$ by substituting each $h_{i_{m}} \neq h_{i}$ with an element of $\PModc{N}$, obtaining the following properties:
 \begin{enumerate}
  \item $\widetilde{f}$ has the form $\widetilde{w}_{0} \circ h_{i}^{\epsilon_{1}} \circ \cdots \circ h_{i}^{\epsilon_{k}} \circ \widetilde{w}_{k}$  with $\epsilon_{j} = \pm 1$ for all $j$.
  \item $\psi_{i}(\widetilde{f}) = \psi_{i}(f)$.
  \item $f(\gamma_{i}) = \widetilde{f}(\gamma_{i})$.
 \end{enumerate}
 
 Note that if $k = 0$, we have that $\widetilde{f} \in \PModc{N}$, and by (2) above we have that $\psi_{i}(f) = 0$. So, we suppose that $k > 0$.
 
 Since for any $w \in \PModc{N}$ we have that $h_{i}^{\epsilon} \circ w \circ h_{i}^{-\epsilon} = w^{\prime}$ for some $w^{\prime} \in \PModc{N}$, we can assume that the form in (1) above satisfies that for all $1 \leq j \leq k-1$, $\epsilon = \epsilon_{j} = \epsilon_{j+1}$. Thus $|\psi_{i}(\widetilde{f})| = k$.
 
 We claim that $k \neq 0$ implies that $\widetilde{f} \notin V$ (and by (3) above $f \notin V$). Note that this claim immediately implies the lemma.
 
 We prove this claim using induction on $|k|$:
 
 If $k = 1$ and $\widetilde{f} \in V$, we have that $h_{i}^{\epsilon} \circ \widetilde{w}_{1} (\gamma_{i}) = \widetilde{w}_{0}^{-1}(\gamma_{i})$. Let $K$ be a compact subsurface that contains $\gamma_{i}$ and $\widetilde{w}_{0}^{-1}(\gamma_{i})$ as essential curves, and contains $\supp(\widetilde{w}_{0})$, $\supp(\widetilde{w}_{1})$ and $h_{i}^{\epsilon}(\supp(\widetilde{w}_{1}))$. Then $\gamma_{i}$, $\widetilde{w}_{0}^{-1}(\gamma_{i})$ and $\widetilde{w}_{1}(\gamma_{i})$ have the same genus in $K$. However, since $\widetilde{w}_{1}(\gamma_{i})$ separates the ends of the support of $h_{i}$, the difference between the genera of $h_{i}^{\epsilon}(\widetilde{w}_{1}(\gamma_{i}))$ and $\widetilde{w}_{1}(\gamma_{i})$ is 1. In particular, $h_{i}^{\epsilon}(\widetilde{w}_{1}(\gamma_{i}))$ and $\widetilde{w}_{0}^{-1}(\gamma_{i})$ cannot have the same genus in $K$, reaching a contradiction. Therefore, if $k = 1$, $\widetilde{f} \notin V$.
 
 Having as induction hypothesis that $k = j \geq 1$ implies that $\widetilde{f} \notin V$, suppose that $k = j+1$ and $\widetilde{f} \in V$. Let $K$ be a compact subsurface that contains $\{\widetilde{w}_{0}^{-1}(\gamma_{i}), \gamma_{i}, \widetilde{w}_{j+1}(\gamma_{i}), h_{i}^{\epsilon} \circ \widetilde{w}_{j+1}(\gamma_{i}), \ldots, h_{i}^{\epsilon} \circ \widetilde{w}_{1} \circ \cdots \circ \widetilde{w}_{j+1}(\gamma_{i})\}$ as essential curves, and contains $\supp(\widetilde{w}_{0}), \ldots, \supp(\widetilde{w}_{j+1})$ and all possible translations of them by the elements $\{h_{i}^{\epsilon}, \widetilde{w}_{j} \circ h_{i}^{\epsilon}, \ldots, h_{i}^{\epsilon}\circ \widetilde{w}_{1} \circ \cdots \circ \widetilde{w}_{j}\}$. By the induction hypothesis, the genera of $\widetilde{w}_{0}^{-1}(\gamma_{i})$ and $\widetilde{w}_{1} \circ \cdots \circ \widetilde{w}_{j+1}(\gamma_{i})$ are different; moreover, the difference between their genus is $j$. Since $\widetilde{w}_{1} \circ \cdots \circ \widetilde{w}_{j+1}(\gamma_{i})$ separates the ends of the support of $h_{i}$, we obtain that the difference between the genus of $h_{i}^{\epsilon} \circ \widetilde{w}_{1} \circ \cdots \circ \widetilde{w}_{j+1}(\gamma_{i})$ and $\widetilde{w}_{1} \circ \cdots \circ \widetilde{w}_{j+1}(\gamma_{i})$ is 1, increasing the difference between the genus of $h_{i}^{\epsilon} \circ \widetilde{w}_{1} \circ \cdots \circ \widetilde{w}_{j+1}(\gamma_{i})$ and $\widetilde{w}_{0}^{-1}(\gamma_{i})$ to $j+1$ (this is because all the $h_{i}$ have the same power). Thus, we reach a contradiction. This finishes the proof of the claim.
\end{proof}

The next step is to extend the homomorphism $\psi_{i}$, for all $0 \leq i < r$, to the closure of $\langle \PModc{N} \cup \{h_{i}\}_{0 \leq i < r} \rangle $. This is done in the following lemma.

\begin{lemma}
 For each $0 \leq i < r$, $\psi_{i}$ extends to a continuous group homomorphism: $$\overline{\psi}_{i}: \overline{\langle \PModc{N} \cup \{h_{i}\}_{0 \leq i <r}\rangle} \to \Z.$$
\end{lemma}
\begin{proof}
 Let $f \in \overline{\langle \PModc{N} \cup \{h_{i}\}_{0 \leq i < r} \rangle} $ and $(f_{j})_{j=0}^{\infty}$ be a sequence of elements of the group $\langle \PModc{N} \cup \{h_{i}\}_{0 \leq i < r} \rangle$ such that 
$$
\lim_{j \rightarrow \infty} f_{j} = f.
$$ 

 Let $N(f(\gamma_{i}))$ a regular neighborhood of $f(\gamma_{i})$ and consider the open set $\bnm{\gamma_{i}, N(f(\gamma_{i}))}$ on $\overline{\langle \PModc{N} \cup \{h_{i}\}_{0 \leq i < r} \rangle}$. Fix $M \geq 0$ such that for all $j \geq M$ we have that $f_{j}\in \bnm{\gamma_{i}, N(f(\gamma_{i}))}$; note this implies that for all $j \geq M$, $f_{j}(\gamma_{i}) = f(\gamma_{i})$. Then, for all $j\geq M$
\begin{align*}
\bnm{\gamma_{i},\, N(f(\gamma_{i}))} \, & = \,  \bnm{\gamma_{i}, \, N(f_{j}(\gamma_{i}))}\\
            \, & = \, f_{j}\cdot\bnm{\gamma_{i}, \, N(\gamma_{i})}.
\end{align*}

In particular  $\bnm{\gamma_{i},\, N(f(\gamma_{i}))}\,=\,f_{M}\cdot\bnm{\gamma_{i},\, N(\gamma_{i})}$. As seen in Lemma \ref{lemma:psiicontinuous}, $\bnm{\gamma_{i},\, N(\gamma_{i})} \cap \langle \PModc{N} \cup \{h_{i}\}_{0 \leq i < r}  \rangle \subset \Ker \psi_{i}$, implying that:
$$
f_{M}\cdot \bnm{\gamma_{i},\, N(\gamma_{i})} \cap \langle \PModc{N} \cup \{h_{i}\}_{0 \leq i < r} \rangle \subset f_{M} \cdot \Ker \psi_{i}.
$$

 Then, for all $j\geq M$ the class $f_{j}$ is in $ f_{M}\cdot \Ker \psi_{i}$. Therefore $\psi_{i}(f_{j}) = \psi_{i}(f_{M})$ for all $j\geq M$, and consequently $(\psi_{i}(f_{j}))_{j=0}^{\infty}$ converges to $\psi_{i}(f_{M})$.
 
 We define $\overline{\psi_{i}}(f) $ as $  \psi_{i}(f_{M})$. The definition of  $\overline{\psi_{i}}$ is independent of the the sequence and thus, restricted to  $\langle \PModc{N} \cup \{h_{i}\}_{0 \leq i < r} \rangle$ is $\psi_{i}$: If $(g_{j})_{j=0}^{\infty}$ is another sequence of $\langle \PModc{N} \cup \{h_{i}\}_{0 \leq i < r} \rangle$ that converges to $f$, then it is eventually contained in $\bnm{\gamma_{i},\, N(f(\gamma_{i}))}$ by the same argument as above. This implies that the sequence $\{g_{j}^{-1}f_{j}\} \subset \langle \PModc{N} \cup \{h_{i}\}_{0 \leq i < r} \rangle$ is eventually contained in $\bnm{\gamma_{i},\, N(\gamma_{i})} \cap \langle \PModc{N} \cup \{h_{i}\}_{0 \leq i < r} \rangle \subset \Ker \psi_{i}$. Thus, the sequences $(\psi_{i}(g_{j}))_{j=0}^{\infty}$ and $\{\psi_{i}(f_{j})\}_{j=0}^{\infty}$ are eventually equal, which implies their limits are the same.

So, we have defined a sequentially continuous function $\overline{\psi_{i}}$. Given that both $\overline{\langle \PModc{N} \cup H \rangle}$ and $\langle h_{i} \rangle$ are metrizable, we have that $\overline{\psi_{i}}$ is continuous. The fact that it is a group homomorphism is a well-known fact of topological groups.
\end{proof}

Finally, we define: 
$$
\begin{array}{rccl}
\overline{\varphi} \,:\, & \PMod{N} = \overline{\langle \PModc{N} \cup \{h_{i}\}_{0 \leq i <r} \rangle} & \longrightarrow & \Z^{r}\\
 & & & \\
 & f & \longmapsto & \left( \overline{\psi_{i}}(f)\right)_{0 \leq i < r}. 
\end{array}
$$

This is obviously a continuous group homomorphism since $\Z^{r}$ has the product topology.

\subsection{The semi-direct product}

\begin{lemma}
 The kernel of $\overline{\varphi}$ is exactly $\overline{\PModc{N}}$.
\end{lemma}
\begin{proof}
 A key observation for the proof of this lemma is the following: for any $i$, $w_{j} \in \PModc{N}$ and $\epsilon = \pm 1$, we have that $$h_{i}^{\epsilon} \circ w_{0} \circ h_{i_{0}}^{n_{0}} \circ \cdots \circ w_{k-1} \circ h_{i_{k}}^{n_{k}} \circ w_{k} \circ h_{i}^{-\epsilon} = \widetilde{w}_{0} \circ h_{i_{0}}^{n_{0}} \circ \cdots \circ \widetilde{w}_{k-1} \circ h_{i_{k}}^{n_{k}} \circ \widetilde{w}_{k},$$ for some $\widetilde{w}_{j} \in \PModc{N}$.
 
 Thus, if $f \in \ker(\psi_{i})$ we can write $f$ as a word in $(\PModc{N} \cup \{h_{j}\}_{0 \leq j < r}) \setminus \{h_{i}\}$.
 
 Now, it is obvious that $\overline{\PModc{N}} \subset \ker(\overline{\varphi})$, so we only need to prove that $\ker(\overline{\varphi}) \subset \overline{\PModc{N}}$.
 
 Let $f \in \ker(\overline{\varphi})$, and $(f_{j})_{j=0}^{\infty} \subset \langle\PModc{N} \cup \{h_i\}_{0\leq i<r}\rangle$ be a sequence converging to $f$; we prove that $f \in \overline{\PModc{N}}$ by exhibiting a sequence $(\widetilde{f}_{j})_{j=0}^{\infty} \subset \PModc{N}$ that converges to $f$.
 
 For each $0 \leq i < r$ we define the following open neighborhoods of $f$:
 \begin{itemize}
  \item $U_{0}:= \ker(\overline{\psi}_{0})$.
  \item For $i >0$, $U_{i} := U_{i-1} \cap \ker(\overline{\psi}_{i})$.
 \end{itemize}
 
 Since $f_{j} \rightarrow f$ as $j \rightarrow \infty$, for each $i$ there exists $N_{i} \geq 0$ such that for all $j \geq N_{i}$ we have that $f_{j} \in U_{i}$. Thus, we may assume that for $N_{i} \leq j < N_{i+1}$ we can write $f_{j}$ as a word in $\PModc{N} \cup \{h_{m}\}_{i+1 \leq m <r}$.
 
 Let $\Sigma_{0} \subset \Sigma_{1} \subset \cdots \subset N$ be a principal exhaustion of $N$ (see Section \ref{sec1} to recall the definition) such that for all $j \geq 0$ and all $i > j$, $\Sigma_{j} \cap \supp(h_{i}) = \varnothing$. For each $j \geq 0$, let $A_{j}$ be a finite set of curves such that $A_{j} \cap \Sigma_{j}$ is a collection of arcs and curves in $\Sigma_{j}$ that satisfies the Alexander method for finite-type surfaces (see Theorem \ref{Thm-AM-finite}), that is for any $g,h \in \Mod{N}$ we have that if $g|_{A_{j}} = h|_{A_{j}}$, then $g|_{\Sigma_{j}} = h|_{\Sigma_{j}}$. For each $j \geq 0$, we define $V_{j} = \cap_{\alpha \in A_{j}} \bnm{\alpha, N(f(\alpha))}$, where $N(\alpha)$ is a regular neighborhood of $\alpha$; since every $A_{j}$ is a finite set, $V_{j}$ is an open set.
 
 Then, for each $i \geq 0$ we define the open set $W_{i} = U_{i} \cap V_{i}$. There exists $M_{i} \geq N_{i} \geq 0$ such that for all $j \geq M_{i}$ we that $f_{j} \in W_{i}$. Note this implies that $f_{j} = w_{j,0} \circ \cdots \circ w_{j,k_{j}}$, where $w_{j,n} \in \PModc{N} \cup \{h_{m}\}_{i+1 \leq m <r}$, and $f_{j}|_{\Sigma_{i}} = f|_{\Sigma_{i}}$.
 
 Now, for each $i \geq 0$ we take $f_{M_{i}} = w_{M_{i},0} \circ \cdots \circ w_{M_{i},k_{M_{i}}}$. We can assume that $w_{M_{i},k_{M_{i}}} = \id$ without any loss in generality, and we do the following algorithm:
 \begin{itemize}
  \item The cycle starts with $j = k_{M_{i}}$, at the end of the instructions decrease $j$ by one and the cycle ends when $j = -1$.
  \begin{itemize}
   \item If $w_{M_{i},j} \in \PModc{N}$: We define $\widetilde{w}_{i,j} := w_{M_{i},j}$.
   \item Else: We have that $w_{M_{i},j} \in \{h_{m}, h_{m}^{-1}\}$ for some $m > i$. Recalling that $A_{i}$ does not separate the ends of $\supp(h_{m})$, by Lemma \ref{lemma:hisubstitutedPmodc} there exists $h \in \PModc{N}$ such that $h|_{\widetilde{w}_{i,j+1} \circ \cdots \circ \widetilde{w}_{i,k_{M_{i}}} (A_{i})} = w_{M_{i},j}|_{w_{M_{i},j+1} \circ \cdots \circ w_{M_{i},k_{M_{i}}}(A_{i})}$. Then we define $\widetilde{w}_{i,j} := h$. This implies that $\widetilde{w}_{i,j} \circ \cdots \circ \widetilde{w}_{i,k_{M_{i}}}|_{A_{i}} = w_{M_{i},j} \circ \cdots \circ w_{M_{i},k_{M_{i}}}|_{A_{i}}$.
  \end{itemize}
  \item Define $\widetilde{f}_{i} := \widetilde{w}_{i,0} \circ \cdots \circ \widetilde{w}_{i,k_{M_{i}}}$.
 \end{itemize}
 
 Note that $\widetilde{f}_{i} \in \PModc{N}$ and $\widetilde{f}_{i}|_{A_{i}} = f_{M_{i}}|_{A_{i}} = f|_{A_{i}}$, which implies that $\widetilde{f}_{i}|_{\Sigma_{i}} = f|_{\Sigma_{i}}$. Hence $\widetilde{f}_{i}\rightarrow f$, and $f \in \overline{\PModc{N}}$.
\end{proof}

So, we have the following short exact sequence of groups
$$
\xymatrix{
1 \ar[r] & \overline{\PModc{N}} \ar[r] & \PMod{N} \ar[r] & \Z^{r} \cong \overline{\langle h_{i} : 0 \leq i <r \rangle} \ar[r] & 1,
}
$$
which naturally splits as the product topology of $\Z^{r}$ coincides with the subgroup topology of $\prod_{i=1}^{r} \, \langle h_{i} \rangle$. Therefore 
$$\PMod{N} = \PModcc{N} \rtimes \prod_{i = 1}^{r} \langle h_{i} \rangle$$
finishing the proof of Theorem \ref{Thm-Semi-direct-prod}.

%%%%%%%%%%%%%%%%%%%%%%%%%%%%%%%%%%%%%%%%%%%%%%%%%%%%%%%%%%%%
%%%%%%%%%%%%%%%%%%%%%%%%%%%%%%%%%%%%%%%%%%%%%%%%%%%%%%%%%%%%
%%%%%%%%%%%%%%%%%%%%%%%%%%%%%%%%%%%%%%%%%%%%%%%%%%%%%%%%%%%%

\subsection{The first integral cohomology group}

As mentioned at the beginning of the section, the argument to prove Corollary \ref{Thm-First-cohom-grp} is analogous to the one presented in \cite{AramayonaPatelVlamis} for the orientable case. For the sake of completeness we present it anyway:

Let $N$ be an infinite-type surface, $\Sigma_{0} \subset \Sigma_{1} \subset \cdots \subset N$ be a principal exhaustion, and $\phi: \PMod{N} \to \Z$ be a homomorphism. We know that $\PModc{N} = \langle \, \bigcup_{0 \leq j < \omega} \PMod{\Sigma_{j}} \, \rangle$. This implies that $\PModc{N}_{ab} = \langle \, \bigcup_{0 \leq j < \omega} \PMod{\Sigma_{j}}_{ab} \, \rangle$. However, by the results of Stukow in \cite{Stukow2010}, for all $j \geq 0$ we have that $\PMod{\Sigma_{j}}_{ab}$ is a torsion group. Thus, $\PModc{N}_{ab}$ is generated by torsion elements, which implies that $\phi|_{\PModc{N}} \equiv 0$.

On one hand, since $\phi$ is a homomorphism from a Polish group to $\Z$, then by Theorem 1 in \cite{Dudley1961} we have that $\phi$ is continuous. Given that the restriction of $\phi$ to $\PModc{N}$ is constantly $0$, we have that $\phi|_{\overline{\PModc{N}}} \equiv 0$.

On the other, we have two possible cases depending on the number of ends of $N$ accumulated by genus: 
\begin{itemize}
 \item If $N$ has at most one end accumulated by genus, by Theorem \ref{Thm-PMod-TopGen-hi} we have that $\overline{\PModc{N}}  = \PMod{N}$. Thus $H^{1}(\PMod{N};\Z) = \mathrm{Hom}(\PMod{N},\Z)$ is trivial.
 \item If $N$ has at least two ends accumulated by genus, by Theorem \ref{Thm-Semi-direct-prod} we have that $\PMod{N} = \PModcc{N} \rtimes \prodhi$.\\ Thus, we have an isomorphism from $H^{1}(\PMod{N};\Z) = \mathrm{Hom}(\PMod{N},\Z)$ to $\mathrm{Hom}(\Z^{r},\Z)$ which, by the results from Specker in \cite{Specker} and the arguments from Blass and Göbel in \cite{BlassGoebel}, is known to be isomorphic to the free abelian group of rank $r$ $\displaystyle \bigoplus_{0\leq i <r} \Z$.
\end{itemize}

%%%%%%%%%%%%%%%%%%%%%%%%%%%%%%%%%%%%%%%%%%%%%%%%%%%%%%%%%%%%
%%%%%%%%%%%%%%%%%%%%%%%%%%%%%%%%%%%%%%%%%%%%%%%%%%%%%%%%%%%%
%%%%%%%%%%%%%%%%%%%%%%%%%%%%%%%%%%%%%%%%%%%%%%%%%%%%%%%%%%%%
%%%%%%%%%%%%%%%%%%%%%%%%%%%%%%%%%%%%%%%%%%%%%%%%%%%%%%%%%%%%
%%%%%%%%%%%%%%%%%%%%%%%%%%%%%%%%%%%%%%%%%%%%%%%%%%%%%%%%%%%%
\bibliographystyle{plain}
\bibliography{bibliography}
\end{document}